\documentclass[11pt,a4paper]{article}
\usepackage{ifthen,latexsym,amssymb,amsmath,bbm,fixmath}
\usepackage[nobysame,initials]{amsrefs}
\usepackage{tikz}
\usepackage{hyperref}
\usepackage{amsthm}
\usepackage{float}
\usepackage{subcaption}

\newcommand{\C}[1]{{\protect\mathcal{#1}}}

\newcommand{\I}[1]{{\mathbbm #1}}
\renewcommand{\O}[1]{\overline{#1}}

\newcommand{\V}[1]{\mathbold{#1}}

\newcommand{\e}{\varepsilon}
\newcommand{\floor}[1]{\lfloor #1\rfloor}

\renewcommand{\mid}{:}
\renewcommand{\ldots}{\hspace{0.9pt}.\hspace{0.3pt}.\hspace{0.3pt}.\hspace{1.5pt}}
\renewcommand{\ge}{\geqslant}
\renewcommand{\le}{\leqslant}
\renewcommand{\geq}{\geqslant}
\renewcommand{\leq}{\leqslant}

\renewcommand{\succeq}{\succcurlyeq}

\newif\ifnotesw\noteswtrue
\newcommand{\comment}[1]{\ifnotesw \textcolor{blue}{ $\blacktriangleright$\ {\sf #1}\ 
  $\blacktriangleleft$ }\fi}
\newcommand{\hide}[1]{}

\newcommand{\beq}[1]{\begin{equation}\label{#1}}
\newcommand{\eeq}{\end{equation}}

\newtheorem{theorem}{Theorem}
\newtheorem{lemma}[theorem]{Lemma}

\newtheorem{corollary}[theorem]{Corollary}

\newtheorem*{claim*}{Claim}

\newcommand{\vc}[1]{\ensuremath{\vcenter{\hbox{#1}}}}
\tikzset{unlabeled_vertex/.style={inner sep=1.9pt, outer sep=0pt, circle, fill}} 
\tikzset{labeled_vertex/.style={inner sep=2.2pt, outer sep=0pt, rectangle, fill=white,draw}} 
\tikzset{edge_color1/.style={color=red,  line width=1.2pt,opacity=0}} 
\tikzset{edge_color2/.style={color=black, line width=0.8pt,opacity=1}} 
\def\outercycle#1#2{ \draw \foreach \x in {0,1,...,#1}{(270-45+\x*360/#2:0.7) coordinate(x\x)};}

\begin{document}

\title{Sharp bounds for decomposing graphs into edges and triangles}

\author{
Adam Blumenthal\thanks{Department of Mathematics, Iowa State University, Ames, IA, E-mail: {\tt ablument@iastate.edu}.}
\and
Bernard Lidick\'{y}\thanks{Department of Mathematics, Iowa State University, Ames, IA, E-mail: {\tt lidicky@iastate.edu}.
Research supported in part by NSF grants DMS-1600390 and DMS-1855653.
}
\and
Yanitsa Pehova\thanks{ Mathematics Institute, University of Warwick, Coventry CV4 7AL, UK, E-mail: {\tt Y.Pehova@warwick.ac.uk}. Supported by the European Research Council (ERC) under the European Union’s Horizon 2020 research and innovation programme (grant agreement No 648509). This publication reflects only its authors' view; the European Research Council Executive Agency is not responsible for any use that may be made of the information it contains.}
\and
Florian Pfender\thanks{Department of Mathematical and Statistical Sciences, University of Colorado Denver, E-mail: {\tt 
Florian.Pfender@ucdenver.edu}. 
Research supported in part by NSF grants  DMS-1600483 and DMS-1855622.
} 
\and
Oleg Pikhurko\thanks{
Mathematics Institute and DIMAP, University of Warwick, Coventry CV4 7AL, UK, E-mail: {\tt O.Pikhurko@warwick.ac.uk}.  Supported by ERC grant 306493, EPSRC grant
EP/K012045/1 and by Leverhulme research project grant RPG-2018-424.
 }
\and 
Jan Volec\thanks{
Department of Mathematics, Faculty of Nuclear Sciences and Physical Engineering, Czech Technical University in Prague, Trojanova 13, 120 00 Prague, Czech Republic, E-mail: {\tt jan@ucw.cz}. Previous affiliation: 
Faculty of Informatics, Masaryk University, Botanick\'{a} 68A, 602 00 Brno, Czech Republic. This project has received funding from the European Union’s Horizon 2020 research and innovation programme under the Marie Skłodowska-Curie grant agreement No. 800607.
This publication reflects only its authors' view; the European Research Council Executive Agency is not responsible for any use that may be made of the information it contains.
}
}

\maketitle

\begin{abstract}
For a real constant $\alpha$, let $\pi_3^\alpha(G)$ be the minimum of twice the number of $K_2$'s plus $\alpha$ times the number of $K_3$'s over all edge decompositions of $G$ into copies of $K_2$ and $K_3$, where $K_r$ denotes the complete graph on $r$ vertices. Let $\pi_3^\alpha(n)$ be the maximum of $\pi_3^\alpha(G)$ over all graphs $G$ with $n$ vertices. 

The extremal function $\pi_3^3(n)$ was first studied by Gy\H{o}ri and Tuza [Decompositions of graphs into complete subgraphs of given order, \emph{Studia Sci.\ Math.\ Hungar.} 22 (1987), 315--320]. 
In a recent progress on this problem, Kr\'al', Lidick\'y, Martins and Pehova [Decomposing graphs into edges and triangles, \emph{Combin.\ Prob.\ Comput.} 28 (2019) 465--472] proved via flag algebras that $\pi_3^3(n)\le (1/2+o(1))n^2$. 
We extend their result by determining the exact value of $\pi_3^\alpha(n)$ and the set of extremal graphs for all $\alpha$ and sufficiently large~$n$. In particular, we show  for $\alpha=3$ that $K_n$ and the complete bipartite graph $K_{\lfloor n/2\rfloor,\lceil n/2\rceil}$ are the only possible extremal examples for large~$n$.
\end{abstract}



\section{Introduction}

In a recent progress on a problem of Gy\H{o}ri and Tuza \cite{Tuza91}, Kr\'al', Lidick\'y, Martins and Pehova~\cite{KralLidickyMartinsPehova19} proved via flag algebras that the edges of any $n$-vertex graph can be decomposed into copies of $K_2$ and $K_3$ whose total number of vertices is at most $(1/2+o(1))n^2$, where $K_r$ denotes the clique on $r$ vertices. The origins of this problem can be traced back to Erd\H{o}s, Goodman and P\'{o}sa \cite{ErdosGoodmanPosa66} who considered the problem of minimising the total number of cliques in an edge decomposition of an arbitrary $n$-vertex graph. They showed the following:

\begin{theorem}[\mdseries Erd\H{o}s, Goodman, P\'{o}sa \cite{ErdosGoodmanPosa66}]\label{thm:erdosgoodmanposa}
The edges of every $n$-vertex graph can be decomposed into at most $\lfloor n^2/4\rfloor$ complete graphs.
\end{theorem}

The only extremal example for this bound is the \emph{(bipartite) Tur\'{a}n graph} $T_2(n):=K_{\lfloor n/2\rfloor,\lceil n/2\rceil}$, where $K_{a,b}$ denotes the complete bipartite graph with part sizes $a$ and~$b$. Moreover, this result still holds if we restrict the sizes of the cliques used in the decomposition to 2 and 3 (that is, single edges and triangles). In a series of papers published independently by Chung \cite{Chung81}, Gy\H{o}ri and Kostochka \cite{Gyori79}, and Kahn \cite{Kahn81}, they proved that in fact something stronger than Theorem~\ref{thm:erdosgoodmanposa} is true, confirming a conjecture by Katona and Tarj\'{a}n:

\begin{theorem}[\mdseries Chung \cite{Chung81}, Gy\H{o}ri and Kostochka \cite{Gyori79}, Kahn \cite{Kahn81}]\label{thm:gyorikostochka}
Every $n$-vertex graph can be edge decomposed into cliques whose total number of vertices is at most $\lfloor n^2/2\rfloor$.
\end{theorem}

For a given graph $G$ on $n$ vertices, let $\pi_k(G)$ be the minimum over all decompositions of the edges of $G$ into cliques $C_1,\ldots,C_\ell$ of size at most $k$ of the sum $|C_1|+|C_2|+\cdots+|C_\ell|$, where $|G|:=|V(G)|$ denotes the \emph{order} of a graph~$G$. Let $\pi_k(n)$ be the maximum of $\pi_k(G)$ over all graphs $G$ with $n$ vertices.
With this notation, the conclusion of the above theorem is that $\min_{k\in \mathbb{N}} \pi_k(n)\le \lfloor n^2/2\rfloor$. In light of Theorem~\ref{thm:gyorikostochka}, Tuza~\cite{Tuza91} conjectured that $\pi_3(n)\le n^2/2+o(n^2)$, and in fact that $\pi_3(n)\le n^2/2+O(1)$. Gy\H{o}ri and Tuza~\cite{GyoriTuza87} showed that $\pi_3(n)\le 9n^2/16$. This was the best known bound until recently, when using the celebrated flag algebra method of Razborov \cite{Razborov07}, Kr\'al', Lidick\'y, Martins and Pehova~\cite{KralLidickyMartinsPehova19} proved the asymptotic version of Tuza's conjecture:

\begin{theorem}[\mdseries Kr\'al', Lidick\'y, Martins and Pehova~\cite{KralLidickyMartinsPehova19}]\label{thm:klmp}
We have $\pi_3(n)\le (1/2+o(1))n^2$ as $n\to\infty$.
\end{theorem}

In this paper we show, by building upon the proof in \cite{KralLidickyMartinsPehova19}, that for all large $n$ it holds in fact $\pi_3(n)\le n^2/2+1$. Moreover, if a graph $G$ of order $n$ attains $\pi_3(n)$ then $G$ is the complete graph $K_n$ or the Tur\'{a}n graph~$T_2(n)$. 

Which of these two graphs is extremal is a matter of divisibility of $n$ by 6. In the case of the Tur\'{a}n graph, we trivially have $\pi_3(T_2(n))=2\lfloor n/2\rfloor \lceil n/2\rceil$, giving $n^2/2$ for even $n$ and $(n^2-1)/2$ for odd $n$. In order to determine $\pi_3(K_n)$, we have to 
determine the maximum number of edge-disjoint triangles in~$K_n$. Clearly, the graph made of their edges is \emph{triangle-divisible}, that is, each vertex has even degree and the total number of edges is divisible by three. It is routine to see that the minimum size of a graph $H$ on $n$ vertices whose complement $\O H$ is triangle-divisible is attained by taking at most one copy of the claw $K_{1,3}$ and a perfect matching on the remaining vertices  for even $n$, and isolated vertices plus at most one copy of the 4-cycle $K_{2,2}$ for odd~$n$. (Note that ${n\choose 2}$ is never equal to 2 modulo~$3$.) In fact, this gives the value of $\pi_3(K_n)$ for all large $n$ by  the following general result (which we will use also inside our proof).

\begin{theorem}[Barber, Kuhn, Lo and Osthus \cite{BarberKuhnLoOsthus16am}]\label{thm:decomp}
For every $\varepsilon > 0$, if $G$ is a triangle-divisible graph of large order $n$ and minimum degree at least $(0.9+\varepsilon)n$, then $G$ has a perfect triangle decomposition.
\end{theorem}
The constant $0.9$ in the minimum degree condition in Theorem~\ref{thm:decomp} comes from the result of Dross~\cite{Dross16siamjdm} on fractional triangle decompostions,
and it was conjectured by Nash-Williams~\cite{NashWilliams70} that it can be replaced by $3/4$.
Very recently, Dukes and Horsley~\cite{DukesHorsley19arxiv}
and Delcourt and Postle~\cite{DelcourtPostle19arxiv} improved the constant to $0.852$ and $(7+\sqrt{21})/14=0.8273...$, respectively.

Let us list the values of $\pi_3$ for the graphs $K_n$ and $T_2(n)$ for large~$n$.


\begin{table}[h]
$$
\renewcommand{\arraystretch}{1.3}
\begin{array}{c|c|c|c}
n \bmod{6} & \mbox{$K_2$'s in an optimal decomposition of }K_n & \pi_3(K_n) & \pi_3(T_2(n)) \\
\hline
0 & \mbox{perfect matching} &  \frac{n^2}{2} & \frac{n^2}{2}\\
1 & \mbox{none} & \binom{n}{2} & \frac{n^2-1}{2} \\
2 & \mbox{perfect matching} & \frac{n^2}{2} & \frac{n^2}{2}\\
3 & \mbox{none} & \binom{n}{2} & \frac{n^2-1}{2} \\
4 & K_{1,3}\mbox{ + perfect matching} & \frac{n^2}{2}+1 & \frac{n^2}{2}\\
5 & C_4 & \binom{n}{2}+4 & \frac{n^2-1}{2}
\end{array}
$$
\caption{Values of $\pi_3(K_n)$ and $\pi_3(T_2(n))$ for large $n$.}
\label{ta:1}
\end{table}

Let us define 
\[
\C E_n := \begin{cases}  
\{T_2(n),K_n\},  & \text{ if }  n \equiv 0,2 \pmod 6,\\
\{T_2(n)\}, & \text{ if }  n \equiv 1,3,5 \pmod 6, \\  
\{K_n\}, & \text{ if }  n \equiv 4 \pmod 6,   
\end{cases}
\]
 and
$$
\ell(n):=\left\{ 
\begin{array}{lll}
n^2/2, & \textrm{ for } n \equiv 0,2 \pmod 6,\\
(n^2-1)/2, & \textrm{ for }  n \equiv 1,3,5 \pmod 6,\\
 n^2/2+1,	& \textrm{ for }  n \equiv 4 \pmod 6.
\end{array}\right.
$$
 Thus, by the calculations of Table~\ref{ta:1}, we have  for all large $n$ that $\C E_n$ consists of those graphs in $\{T_2(n),K_n\}$ which maximise $\pi_3$ while $\ell(n)$ is this maximum value.  

Clearly, $\ell(n)$ is a lower bound on $\pi_3(n)$
for large~$n$. Our main result is that this is equality.

\begin{theorem}\label{th:1} There exists $n_0\in \mathbb{N}$ such that for all $n\ge n_0$, we have $\pi_3(n)=\ell(n)$ and the set of $\pi_3(n)$-extremal graphs up to isomorphism is exactly $\C E_n$. 
\hide{This gives
\[
\pi_3(n) =\left\{ 
\begin{array}{lll}
n^2/2 & \textrm{ for } n \equiv 0,2 \bmod 6 & \textrm{attained only by }T_2(n)\textrm{ and }K_n,\\
(n^2-1)/2 & \textrm{ for }  n \equiv 1,3,5 \bmod 6 & \textrm{attained only by }T_2(n),\\
n^2/2+1	& \textrm{ for }  n \equiv 4 \bmod 6 & \textrm{attained only by }K_n.
\end{array}\right.
\]
}
\end{theorem}



A simple corollary of Theorem~\ref{th:1} is an affirmative answer to a question of Pyber~\cite{Pyber92}, see also~\cite[Problem~45]{Tuza91}, for sufficiently large $n$.
A \emph{covering} of a graph $G$ is a collection of subgraphs of $G$ such that every edge of $G$ appears in at least one subgraph. (For comparison, a decomposition requires that every edge appears in exactly one subgraph.)

\begin{corollary}
There exists $n_0\in \mathbb{N}$ such that for all $n\ge n_0$, the edge set of every $n$-vertex graph can be covered with triangles  and edges so that the sum of their orders is at most $\lfloor n^2/2 \rfloor$.
\end{corollary}
\begin{proof}
Theorem~\ref{th:1} directly implies the corollary unless $n \equiv 4 \pmod 6$ and the graph under consideration is $K_n$.
So assume that $n \equiv 4 \pmod 6$. 
Denote the vertices of $K_n$ by $v_1,\ldots,v_n$.
Recall that an optimal decomposition for $K_n$ is obtained by taking edges $v_1v_2, v_1v_3, v_1v_4$ and $v_iv_{i+1}$ for all odd $i$ with $5 \leq i \leq n-1$.
The rest of the graph becomes triangle-divisible and Theorem~\ref{thm:decomp} can be applied. This gives a decomposition of cost $n^2/2+1$.
A covering of cost at most $n^2/2$ can be obtained from this decomposition by replacing edges  $v_1v_2$ and $v_1v_3$ by a triangle $v_1v_2v_3$.
(Notice that the pair $v_2v_3$ is covered by two triangles in the resulting covering.)
\end{proof}
 
 We also study an extension of Theorem~\ref{th:1}, where we consider decompositions into $K_2$'s and $K_3$'s but we modify the cost of $K_3$'s to be $\alpha$ (with the cost of $K_2$ still being 2).
 The minimum over all costs of such decompositions of a graph $G$ is denoted by $\pi_3^\alpha(G)$. 
 The maximum value of $\pi_3^\alpha(G)$ over all $n$-vertex graphs $G$ is denoted by  $\pi_3^\alpha(n)$.
 Notice that $\pi_3^3(G) = \pi_3(G)$ and $\pi_3^3(n) = \pi_3(n)$.
 Denote $K_n$ without one edge by $K_n^-$ and $K_n$ without a matching of size two by $K_n^{=}$. Then the following result holds.
 
 \begin{theorem}\label{thm:alpha}
 For every real $\alpha$ exists $n_0\in \mathbb{N}$ such that every $\pi_3^\alpha$-extremal graph $G$ with $n\ge n_0$ vertices 
satisfies the following (up to isomorphism).
\begin{itemize}
\item If $\alpha<3$, then  $G= T_2(n)$;
\item if $\alpha=3$ then Theorem \ref{th:1} applies;
\item if $3 < \alpha < 4$ and $n\equiv 0,2,4,5\pmod 6$, then $G=K_n$;
\item if $3 < \alpha < 4$ and $n\equiv 1,3\pmod 6$, then $G=K_n^=$;
\item if $\alpha=4$ and $n\equiv 1,3\pmod 6$, then $G\in \{K_n, K_n^-, K_n^=\}$ and, moreover, the  three listed graphs are all $\pi_3^\alpha$-extremal;
\item if $\alpha=4$ and $n\equiv 0,2,4,5\pmod 6$, then $G=K_n$;
\item if $4 < \alpha$, then $G=K_n$.
\end{itemize}
 \end{theorem}

\medskip 
 This paper is organised as follows. In Section~\ref{sec:T3} we give an outline of the proof of Theorem~\ref{thm:klmp} from~\cite{KralLidickyMartinsPehova19} that we build on.
 Theorem~\ref{th:1} is proved in Section~\ref{sec:proof}.
 Extension for other weights of triangles is in Section~\ref{sec:alpha}.
Some related results are mentioned in Section~\ref{sec:related}.

\textbf{Notation.} We follow standard graph theory notation (see e.g.\ \cite{Bollobas:mgt}). 

For a graph $G$, we denote the set neighbours of $x\in V(G)$ by $\Gamma_{G}(x)$ (or just $\Gamma(x)$ when $G$ is understood) and the number of edges in a set $B\subseteq E(G)$ incident with $x$ by $d_B(x)$. We denote by $K[V_1,V_2]$ the complete bipartite graph with vertex partition $(V_1,V_2)$. The term $[X,Y]$-edges refers to edges $xy\in E(G)$ such that $x\in X$ and $y\in Y$. We write $[x,Y]$-edges as a short-hand for $[\{x\},Y]$-edges. 

Let $t_2(n):=|E(T_2(n))|$ be the number of edges in the Tur\'an graph~$T_2(n)$. 
Recall that $t_2(n)=\lfloor n^2/4 \rfloor$. By a  \emph{cherry} we mean a path with 2 edges.

\renewcommand{\cong}{=}
We consider graphs up to isomorphism; in particular, we write $G\cong H$ to denote that $G$ and $H$ are isomorphic graphs.

\section{Outline of the proof of Theorem~\ref{thm:klmp} from~\cite{KralLidickyMartinsPehova19}}\label{sec:T3}

In this section we give a short outline of the proof of \cite[Lemma 5]{KralLidickyMartinsPehova19}, which was a key step in proving
$\pi_3(n)\le n^2/2+o(n^2)$ and is a
starting point of our argument towards Theorem \ref{th:1}. For an $n$-vertex graph $G$ and each $i\in\mathbb{N}$, let $K_i(G)$ be the set of all $i$-cliques in $G$. Let $\pi_{3,f}(G)$ be the minimum of 
$$
 2\sum_{xy\in K_2(G)} c(xy)+3 \sum_{xyz\in K_3(G)}c(xyz)
 $$ over \emph{fractional $\{K_2,K_3\}$-decompositions} $c$ of $E(G)$, that is, over maps $c:K_2(G)\cup K_3(G)\to [0,1]$ such that for every edge $xy\in E(G)$ we have $c(xy)+\sum_{z:xyz\in K_3(G)} c(xyz)\ge 1$. Of course, $\pi_{3,f}(G)\le \pi_3(G)$. By a result of Haxell and R\"{o}dl \cite{HaxellRodl01} or a more general version by Yuster~\cite{Yuster05}, it also holds that $\pi_3(G)\le \pi_{3,f}(G)+o(n^2)$. So, to show that $\pi_3(G)\leq n^2/2+o(n^2)$, it suffices to consider the fractional equivalent $\pi_{3,f}(G)$.
 
\begin{lemma}\label{lem:flag}
Let $G$ be an $n$-vertex graph. Then
\[\binom{n}{7}^{-1}\sum_{W\in \binom{V(G)}{7}}\pi_{3,f}(G[W])\leq 21+o(1)\]
where the sum is taken over 7-vertex subsets $W$ of~$V(G)$.
\end{lemma}
\begin{proof}[Outline of proof.]
Let $M$ be the following positive semi-definite matrix
\begin{center}
$M:= \frac{1}{12\cdot 10^{9}}$
\scalebox{0.65}{
$\begin{pmatrix}
 1800000000 & 2444365956 &  640188285 & -1524146769 & 1386815580 & -732139362  & -129387078 \\
 2444365956 & 4759879134 & 1177441152 & -1783771230 & 2546923788 & -1397639394 & -143552208 \\
  640188285  & 1177441152 &  484273772 & -317303211&  1038156300  & -591902130   & -6783162 \\
-1524146769 & -1783771230 & -317303211 & 1558870290 & -651906630  & 305728704  & 154602378 \\
 1386815580 & 2546923788 & 1038156300 & -651906630 &  2285399634 & -1283125950  & -10755036 \\
 -732139362 & -1397639394 & -591902130 &  305728704 & -1283125950  & 734039016   & -1621938 \\
 -129387078  & -143552208  &  -6783162  & 154602378  & -10755036   & -1621938  &  23860164  
\end{pmatrix}$} $\succeq 0$
\end{center}
and let $\overrightarrow F:=(F_1,\dots,F_7)$ be the following vector of rooted graphs, each having 4 vertices with the root denoted by the white square:
\begin{align*}
\vec F = \left(
\vc{
{\begin{tikzpicture}\outercycle{5}{4}
\draw[edge_color1] (x0)--(x1);\draw[edge_color1] (x0)--(x2);\draw[edge_color1] (x0)--(x3);  \draw[edge_color1] (x1)--(x2);\draw[edge_color1] (x1)--(x3);  \draw[edge_color1] (x2)--(x3);    
\draw (x0) node[labeled_vertex]{};\draw (x1) node[unlabeled_vertex]{};\draw (x2) node[unlabeled_vertex]{};\draw (x3) node[unlabeled_vertex]{};
\end{tikzpicture}}
,
{\begin{tikzpicture}\outercycle{5}{4}
\draw[edge_color1] (x0)--(x1);\draw[edge_color1] (x0)--(x2);\draw[edge_color1] (x0)--(x3);  \draw[edge_color1] (x1)--(x2);\draw[edge_color1] (x1)--(x3);  \draw[edge_color2] (x2)--(x3);    
\draw (x0) node[labeled_vertex]{};\draw (x1) node[unlabeled_vertex]{};\draw (x2) node[unlabeled_vertex]{};\draw (x3) node[unlabeled_vertex]{};
\end{tikzpicture}}
,
{\begin{tikzpicture}\outercycle{5}{4}
\draw[edge_color1] (x0)--(x1);\draw[edge_color1] (x0)--(x2);\draw[edge_color2] (x0)--(x3);  \draw[edge_color1] (x1)--(x2);\draw[edge_color2] (x1)--(x3);  \draw[edge_color2] (x2)--(x3);    
\draw (x0) node[labeled_vertex]{};\draw (x1) node[unlabeled_vertex]{};\draw (x2) node[unlabeled_vertex]{};\draw (x3) node[unlabeled_vertex]{};
\end{tikzpicture}}
,
{\begin{tikzpicture}\outercycle{5}{4}
\draw[edge_color2] (x0)--(x1);\draw[edge_color2] (x0)--(x2);\draw[edge_color2] (x0)--(x3);  \draw[edge_color1] (x1)--(x2);\draw[edge_color1] (x1)--(x3);  \draw[edge_color1] (x2)--(x3);    
\draw (x0) node[labeled_vertex]{};\draw (x1) node[unlabeled_vertex]{};\draw (x2) node[unlabeled_vertex]{};\draw (x3) node[unlabeled_vertex]{};
\end{tikzpicture}}
,
{\begin{tikzpicture}\outercycle{5}{4}
\draw[edge_color1] (x0)--(x1);\draw[edge_color2] (x0)--(x2);\draw[edge_color2] (x0)--(x3);  \draw[edge_color1] (x1)--(x2);\draw[edge_color2] (x1)--(x3);  \draw[edge_color2] (x2)--(x3);    
\draw (x0) node[labeled_vertex]{};\draw (x1) node[unlabeled_vertex]{};\draw (x2) node[unlabeled_vertex]{};\draw (x3) node[unlabeled_vertex]{};
\end{tikzpicture}}
,
{\begin{tikzpicture}\outercycle{5}{4}
\draw[edge_color1] (x0)--(x1);\draw[edge_color2] (x0)--(x2);\draw[edge_color2] (x0)--(x3);  \draw[edge_color2] (x1)--(x2);\draw[edge_color2] (x1)--(x3);  \draw[edge_color1] (x2)--(x3);    
\draw (x0) node[labeled_vertex]{};\draw (x1) node[unlabeled_vertex]{};\draw (x2) node[unlabeled_vertex]{};\draw (x3) node[unlabeled_vertex]{};
\end{tikzpicture}}
,
{\begin{tikzpicture}\outercycle{5}{4}
\draw[edge_color1] (x0)--(x1);\draw[edge_color2] (x0)--(x2);\draw[edge_color2] (x0)--(x3);  \draw[edge_color2] (x1)--(x2);\draw[edge_color2] (x1)--(x3);  \draw[edge_color2] (x2)--(x3);    
\draw (x0) node[labeled_vertex]{};\draw (x1) node[unlabeled_vertex]{};\draw (x2) node[unlabeled_vertex]{};\draw (x3) node[unlabeled_vertex]{};
\end{tikzpicture}}
}
\right).
\end{align*}

Take any graph $G$ of order $n\to\infty$. For $w\in V(G)$, let $\V v_{G,w}\in\I R^7$ denote the column vector whose $i$-th component is $p(F_i,(G,w))$,
the density of the $1$-flag $F_i$ in the rooted graph $(G,w)$, which is $G$ with the vertex $w$ designated as the root.

It was shown in~\cite{KralLidickyMartinsPehova19} that
\beq{eq:1}
\frac1{{n\choose 7}}\sum_{W\in {V(G)\choose 7}} \pi_{3,f}(G[W])+\frac1n \sum_{w\in V(G)} \V v_{G,w}^T M\V v_{G,w}
\le 21
+o(1).
\eeq
 Namely, if we re-write the left-hand size as a linear combination $\sum_H c_H p(H,G)$, where $H$ ranges over all $7$-vertex unlabelled graphs and $p(H,G)$ is the density of $H$ in $G$, then each
coefficient $c_H$ is at most $21$. Since $\sum_H p(H,G)=1$, the claimed inequality~\eqref{eq:1} follows.

In particular, since $M$ is positive semi-definite, the quantity $\frac1n \sum_{w\in V(G)} \V v_{G,w}^T M\V v_{G,w}$ is always non-negative, yielding the result.
\end{proof}

The main result of \cite{KralLidickyMartinsPehova19} that $\pi_3(n)\le n^2/2+o(n^2)$ now follows directly from Lemma \ref{lem:flag}.

\begin{proof}[Proof of Theorem~\ref{thm:klmp}.]
Let $G$ be any graph of order $n\to\infty$. As mentioned before, $\pi_3(G)\leq \pi_{3,f}(G)+o(n^2)$. Also, we have
$$\binom{n}{2}^{-1}\pi_{3,f}(G)\le \binom{7}{2}^{-1}\binom{n}{7}^{-1}\sum_{W\in {V(G)\choose 7}} \pi_{3,f}(G[W]),$$ 
by averaging optimal fractional decompositions of all 7-vertex induced subgraphs. Combining this inequality with Lemma \ref{lem:flag} immediately gives that $\pi_3(G)\le (1/2+o(1))n^2$.
\end{proof}

\section{Proof of Theorem~\ref{th:1}}\label{sec:proof}

We use the so-called \emph{stability approach}, where the first step is to describe the approximate structure of all almost $\pi_3$-extremal graphs of order $n\to\infty$ within $o(n^2)$ adjacencies. Namely, our Corollary~\ref{co:stab}
will show that every such graph is close to $K_n$ or~$T_2(n)$. 

For this purpose, we start by showing that all almost $\pi_3$-extremal graphs contain almost no copies of the three graphs in Figure \ref{fig:H} (which are obtained by taking the unlabelled versions of the corresponding graphs in $\overrightarrow{F}$). This is achieved by the following lemma that builds on the results from~\cite{KralLidickyMartinsPehova19}.

\begin{figure}[H]
\begin{center}
\vc{
{\begin{tikzpicture}\outercycle{5}{4}
\draw[edge_color1] (x0)--(x1);\draw[edge_color1] (x0)--(x2);\draw[edge_color1] (x0)--(x3);  \draw[edge_color1] (x1)--(x2);\draw[edge_color1] (x1)--(x3);  \draw[edge_color2] (x2)--(x3);    
\draw (x0) node[unlabeled_vertex]{};\draw (x1) node[unlabeled_vertex]{};\draw (x2) node[unlabeled_vertex]{};\draw (x3) node[unlabeled_vertex]{};
\draw (0,-1) node {$H_2$};
\end{tikzpicture}}}
\hskip 3em
\vc{\begin{tikzpicture}\outercycle{5}{4}
\draw[edge_color1] (x0)--(x1);\draw[edge_color2] (x0)--(x2);\draw[edge_color2] (x0)--(x3);  \draw[edge_color1] (x1)--(x2);\draw[edge_color2] (x1)--(x3);  \draw[edge_color2] (x2)--(x3);    
\draw (x0) node[unlabeled_vertex]{};\draw (x1) node[unlabeled_vertex]{};\draw (x2) node[unlabeled_vertex]{};\draw (x3) node[unlabeled_vertex]{};
\draw (0,-1) node {$H_5$};
\end{tikzpicture}}
\hskip 3em
\vc{\begin{tikzpicture}\outercycle{5}{4}
\draw[edge_color1] (x0)--(x1);\draw[edge_color2] (x0)--(x2);\draw[edge_color2] (x0)--(x3);  \draw[edge_color2] (x1)--(x2);\draw[edge_color2] (x1)--(x3);  \draw[edge_color2] (x2)--(x3);    
\draw (x0) node[unlabeled_vertex]{};\draw (x1) node[unlabeled_vertex]{};\draw (x2) node[unlabeled_vertex]{};\draw (x3) node[unlabeled_vertex]{};
\draw (0,-1) node {$H_7$};
\end{tikzpicture}}
\end{center}
\caption{Graphs $H_2$, $H_5$, and $H_7$.}\label{fig:H}
\end{figure}
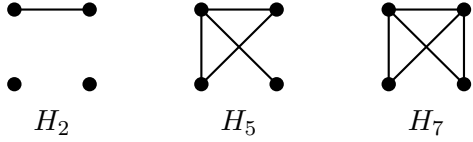

\begin{lemma}\label{lm:stab} For every $c>0$ there exist $\e>0$ and $n_0\in \mathbb{N}$ such that for all $n\ge n_0$, if $G$ is a graph of order $n$ with $\pi_3(G)\ge(1/2-\e)n^2$, then $G$ has at most $c \binom{n}{4}$ copies of each of the graphs $H_2:=(\{a,b,c,d\},\{ab\})$, $H_5:=(\{a,b,c,d\},\{ab,bc,ac,ad\})$ and $H_7:=(\{a,b,c,d\},\{ab,bc,ac,bd,ad\})$ from Figure~\ref{fig:H}.
\end{lemma}

\begin{proof} Given $c>0$, let $\e\gg 1/n_0>0$ be sufficiently small. Let $G$ be a graph as in the lemma. Let $M$ and $\overrightarrow{F}$ be as in the proof of Lemma \ref{lem:flag}.

First, the rank of the matrix $M$ is 6 with $\V v=(1,0,3,1,0,3,0)$ being the only zero eigenvector. (Thus all others eigenvalues of $M$ are strictly positive by~$M\succeq 0$.) 

Second, by the almost optimality of $G$ and the fact that each term in the left-hand side of~\eqref{eq:1} is non-negative, we have that
 \beq{eq:2}
 \sum_{w\in V(G)} \V v_{G,w}^T M\V v_{G,w}=o_\e(n).
 \eeq
 
We now show that $G$ must contain few copies of the graphs $H_2$, $H_5$ and $H_7$. Suppose, for contradiction, that $G$ contains at least $c\binom{n}{4}$ copies of $H_2$. Then, by a simple double-counting argument we have that at least $cn/4$ vertices in $G$ contain at least $c\binom{n}{3}/4$ copies of the rooted flag~$F_2$. In particular, the second coordinate of at least $cn/4$ of the vectors $\V v_{G,w}$ is at least $c/4$. For each such vector $\V u$, let $\V u':=\V u/\|\V u\|_2$ be the scalar multiple of $\V u$ of $\ell^2$-norm 1. Since $\|\V u\|_2\le \sqrt 7$, we have that its second coordinate $\V u'_2$ is at least $c/4\sqrt 7$. The scalar product of $\V u'$ and the $\ell^2$-normalised zero eigenvector $\V v/\sqrt{20}$ (whose second coordinate is 0) is at most $\sqrt{1-(c/4\sqrt{7})^2}$. Thus  the projection of $\V u$ on the orthogonal complement $L=\V v^\perp$ of the zero eigenspace of $M$ has $\ell^2$-norm at least $c/4\sqrt{7}$. Thus $\V u^T M\V u\ge \lambda_2 (c/4\sqrt{7})^2$, where $\lambda_2>0$ is the smallest positive eigenvalue of $M$ (in fact, one can check with computer that $\lambda_2=0.0005228...$). Thus, we have that the left-hand side of \eqref{eq:2} in which each term is non-negative by $M\succeq 0$ is at least $(cn/4) \times \lambda_2(c/4\sqrt{7})^2=\Omega(n)$, a contradiction. 

The analogous argument shows that the densities of $H_5$ and $H_7$ in $G$ are also at most~$c$.
\end{proof}

Let us say that two graphs $G_1$ and $G_2$ of the same order are \emph{$k$-close in the edit distance} (or simply \emph{$k$-close}) if there is a relabelling of the vertices of $G_2$ so that $|E(G_1)\triangle E(G_2)|\leq k$. In other words, we can make $G_1$ and $G_2$ isomorphic by changing at most $k$ adjacencies.

\begin{corollary}\label{co:stab} For every $\delta>0$ there exists $n_1\in \mathbb{N}$ such that if $G$ is a graph of order $n\ge n_1$ with $\pi_3(G)\ge \ell(n)-n^2/n_1$, then $G$ 
	is $\delta n^2$-close in edit distance to $K_n$ or to $T_2(n)$.
\end{corollary}

\begin{proof} Given any $\delta>0$, choose sufficiently small constants $\delta \gg c\gg 1/n_1>0$. Take any graph $G$ on $n\ge n_1$ vertices such that $\pi_3(G)\ge \ell(n)-n^2/n_1$.

By Lemma~\ref{lm:stab} and the Induced Removal Lemma~\cite{AlonFischerKrivelevichSzegedy00}, $G$ can be made $\{H_2,H_5,H_7\}$-free by changing at most $c n^2$ adjacencies. Denote this new graph by $G'$ and note that 
$\pi_3(G')\ge \pi_3(G) - 2c n^2$. By $c\ll \delta$, it is enough to show that $G'$ is $\delta n^2/2$-close to $K_n$ or~$T_2(n)$.
	
Let us show that $G'$ is either triangle-free, or the disjoint union of at most two cliques. Indeed, if some vertices $a,b,c$ span a triangle in $G'$ then, by the $\{H_5,H_7\}$-freeness of $G$, all the remaining vertices of $G'$ have either no or three neighbours among $\{a,b,c\}$. Let $A_0$ be the set of vertices in $G'\backslash \{a,b,c\}$ which see none of $\{a,b,c\}$, and let $A_3$ be the set of vertices which see all of $\{a,b,c\}$. Then $A_3$ is a clique because $G'$ is $H_7$-free. The set $A_0$ is also a clique because $G'$ is $H_2$-free. Also, no pair $xy$ in $A_3\times A_0$ can be an edge as otherwise e.g.\ the 4-set $\{a,b,x,y\}$ spans a copy of $H_5$ in $G$. It follows that $G$ is the disjoint union of the cliques on $A_0$ and $A_3\cup\{a,b,c\}$, as required.

Now, if $G'$ is triangle-free, then $e(G')=\pi_3(G')/2\ge \ell(n)/2-n^2/n_1-2c n^2\ge t_2(n)-3cn^2$. Thus, by the stability result for Mantel's theorem 
by Erd\H os~\cite{Erdos67a} and Simonovits~\cite{Simonovits68}, 
the graph $G'$ must indeed be $\delta n^2/2$-close in edit distance to $T_2(n)$.

Otherwise, $G'$ is the disjoint union of two cliques. Let us show that one of them has size at most $\delta n/2$. Indeed, otherwise
$G'$ has a triangle packing covering all but at most $n/2+2$ edges  by 
Theorem~\ref{thm:decomp}, meaning that $\pi_3(G')\le e(G')+n/2+2$. Also, $e(G')$ is
maximum when clique sizes are as far apart as possible. Thus, by the lower bound on $\pi_3(G)\le \pi_3(G')+2c n^2$, we conclude that e.g.\ $\ell(n)-3c n^2\le {\delta n/2\choose 2}+{(1-\delta/2)n\choose 2}$, leading to a contradiction to our choice of constants. Therefore, $G'$ is at most $n\cdot \delta n/2$ adjacency edits away from $K_n$, as desired.
\end{proof}

The key steps in proving Theorem~\ref{th:1} are Lemmas~\ref{lm:keyT2n}--\ref{lm:Kn}.

\begin{lemma}\label{lm:keyT2n} 
There exist constants $\delta>0$ and $n_1\in \mathbb{N}$ such that, among all graphs on $n \geq n_1$ vertices which are $\delta n^2$-close to $T_2(n)$, the maximiser of $\pi_3$ is $T_2(n)$. 
\end{lemma}

\begin{proof}
Choose sufficiently small $\e\gg \delta\gg 1/n_1>0$. Let $G$ be an arbitrary graph with $n\ge n_1$ vertices which is $\delta n^2$-close to~$T_2(n)$. 
We will show that $\pi_3(G)\le \pi_3(T_2(n))$ with equality if and only if $G= T_2(n)$.
In fact, this claim can be directly derived from the result of Gy\H ori~\cite[Theorem~1]{Gyori88} that a graph with $n$ vertices and $t_2(n)+k$ edges,  where $n\to\infty$ and $k=o(n^2)$,
has at least $k-O(k^2/n^2)$ edge-disjoint triangles.  More specifically, for each $\e>0$ there exists $\delta>0$ and $n_0 \in \mathbb{N}$ such that every graph with $n\ge n_0$ vertices and $t_2(n)+k$ edges, where $k\leq \delta n^2$, has at least $k-\e k^2/n^2$ edge-disjoint triangles. (See also~\cite[Theorem~1]{Gyori91} for a generalisation of this to $r$-cliques for any fixed $r\ge 3$.) Since $G$ is $\delta n^2$-close to $T_2(n)$, it must have at most $t_2(n)+\delta n^2$ edges. From this and  $1/n\ll \delta\ll \e\ll 1$, we have that, for $k:=e(G)-t_2(n)$, 
$$
\pi_3(G)\leq 2(t_2(n))+k)-3(k-\e k^2/n^2)=2t_2(n)-k(1-3\e k/n^2)\leq 2t_2(n).
$$ 
Clearly, if equality is achieved then $k=0$, that is, $e(G)=t_2(n)$; furthermore, $G$ must be triangle-free and thus $G= T_2(n)$, as required.
\end{proof}

Next, we need to analyse graphs that are close to $K_n$. 
If $n\equiv 1,3\pmod 6$, then let $\C E'_n$ consist of those graphs which are obtained from $K_n$ by removing a matching of size $m\equiv 2\pmod 3$;
otherwise let $\C E'_n:=\{K_n\}$. Also, define
$$
w(n):=\left\{\begin{array}{ll} n/2,& n\equiv 0,2\pmod 6,\\
2,&n\equiv 1,3\pmod 6,\\
n/2+1,&n\equiv 4\pmod 6,\\
4,& n\equiv 5 \pmod 6.\end{array}\right.
$$
 Using Theorem~\ref{thm:decomp} and the calculation for $K_n$ described in Table~\ref{ta:1}, one can show that  $\pi_3(G)={n\choose 2}+w(n)$ for all large $n$ and every $G\in\C E'_n$. We are going to show that these are exactly the extremal graphs among those close to~$K_n$.
 It is more convenient to do first the case when we have some bound on the minimum degree of a graph and then derive the general case (in a separate Lemma~\ref{lm:Kn}).

\begin{lemma}\label{lm:KnDeg}
	There exist constants $\delta>0$ and $n_0\in \mathbb{N}$ such that the following holds.
	Let $G$ be a graph  on $n \geq n_0$ 
 vertices 
with minimum degree  at least $n/8$ such
	that $G$ is $\delta n^2$-close to $K_n$ and $\pi_3(G)\ge {n\choose 2}+w(n)$.
	Then $G\in\C E'_n$.
\end{lemma}

\begin{proof} 
Choose small constants in the following order: $c \gg \delta \gg 1/n_0> 0$.
Suppose that $G$ is a graph of order $n\ge n_0$ as  in the statement of the lemma.
Let $w:=w(n)$. 
	
Let $U:=\{v\in V(G):d_G(v)\le (1-c)n\}$. Then 
	\[
\frac{|U|cn}{2}\le e(\,\overline{G}\,)\le \delta n^2,
\]
	and so $|U|\le \frac{2\delta}{c}n$. Denote $W:=V(G)\setminus U$, and let $S:=\{v\in W: d_G(v)\mbox{ is odd}\}$. Let $M$ be a set of edges forming a maximum matching in $G[S]$, and denote $X:=S\setminus V(M)$. Then $X$ is an independent set and thus ${|X|\choose 2}\le \delta n^2$, which implies that rather roughly
	\begin{equation}\label{eq:|X|}
	|X|<cn.
	\end{equation} 
	Moreover, for every edge $yz\in M$ and any two distinct  vertices $y',z'\in X$, at most one of  $yy'$ and $zz'$ can be an edge of $G$ (otherwise $y'yzz'$ is an augmenting path contradicting the maximality of $M$). It follows that, if $|X|\not=1$, then for every edge $yz\in M$ there are at least $|X|$ edges missing between $yz$ and $X$. Let $Y_W$ denote the set of missing edges in $G[W]$. Thus
	\begin{equation}\label{eq:yw}
|Y_W|\ge \binom{|X|}{2}+|M|(|X|-\I 1_{|X|=1}),
	\end{equation}
where the indicator function $\I 1_{|X|=1}$ is 1 if $|X|=1$ and is 0 otherwise. Moreover, the set $Y_U$ of missing edges in $G$ with at least one endpoint in $U$ satisfies
	\begin{equation}\label{eq:yu}
	|Y_U|\ge cn|U|-\binom{|U|}{2}
	\end{equation}
	by the definition of $U$. Note that $e(G)=\binom{n}{2}-|Y_W|-|Y_U|$.
See Figure~\ref{fig:decom} for a sketch ot $Y_W$ and $Y_U$.	
	
	We now build a decomposition $\C D$ of $G$ into edges and triangles, starting with $\C D=\emptyset$. If we add edges/triangles to $\C D$, we regard them as removed from~$E(G)$. It is convenient to split our argument into the following two cases.

\medskip \noindent\textbf{Case 1:} $U\not=\emptyset$ or $S=\emptyset$.\medskip

In this case, our procedure for constructing $\C D$ is as follows. 
	
	\begin{description}
		
		\item[Step 1:] Add the following to $\C D$ as $K_2$'s: the edges of 
the matching $M$ and the edges of some $\lfloor |X|/2\rfloor$ cherries with distinct endpoints in $X$ such that their middle points are pairwise distinct. 
		
		
		\item[Step 2:] For each $u \in U$, one at a time, add  to $\C D$ a maximum set of edge-disjoint $K_3$'s containing $u$ and two vertices from $W$. Add  all remaining edges incident to vertices in $U$ as~$K_2$'s to~$\C D$.
			
		
		\item[Step 3:] \textbf{(a)} Let $S'\subseteq V(G)$ be the set of vertices with odd degree after Step $2$. Add to $\C D$ the edges of some $|S'|/2$ cherries with distinct endpoints in $S'$ such that their middle points are pairwise distinct. \\
\textbf{(b)}   If the number of remaining edges is not divisible by $3$, then fix this by adding to $\C D$ (as single edges) the edge set of some cycle of length $4$ or $5$.

		
		\item[Step 4:] 
		      Add a perfect triangle decomposition of the remaining edges to $\C D$. 
	\end{description}

For $i \in \{1,2,3\}$, let $Z_i$ be the set of edges that are added to $\C D$ in Step~$i$ as copies of $K_2$. See Figure~\ref{fig:decom} for some illustrations of the above steps.

	\usetikzlibrary{calc}
	\tikzset{
		vtx/.style={inner sep=1.2pt, outer sep=0pt, circle, fill=black,draw=black}, 
		vtx_gray/.style={inner sep=1.2pt, outer sep=0pt, circle, fill=gray,draw}, 
		vtx_mini/.style={inner sep=1.2pt, outer sep=0pt, circle, fill=black,draw=black}, 
		zw/.style={color=blue}, 
		zu/.style={color=red}, 
		yw/.style={color=blue,dashed}, 
		yu/.style={color=red,dashed}, 
	} 
	
	\begin{figure}[ht]
		\begin{center}
			\begin{tikzpicture}[scale=1.5]
			\draw (0,0) ellipse (0.5cm and 0.8cm);
			\draw[rounded corners=20pt] (1,-1) rectangle (3,1.2) ;
			\draw (0,-0.4) node[vtx,label=left:$ $](v){};
			\draw (0,0.4) node[vtx,label=left:$ $](u){};
			\foreach \x in {-2,-1}{
				\draw[yu] (u) -- (1.5,0.4*\x) node[vtx_mini]{} ++(0,0.2) node[vtx_mini]{} -- (u);
			}
			\draw[yu] (u) -- (1.5,0.4*0) node[vtx_mini]{};
			
			\foreach \x in {2,1}{
				\draw[yu] (v) -- (1.5,0.4*\x) node[vtx_mini]{} ++(0,0.2)node[vtx_mini]{} -- (v);
			}
			\draw[yu] (v) -- (1.5,0.4*0.5) node[vtx_mini]{};

			\foreach \x in {-0.1,0.1}{
				\draw[yu] (-0.2,\x)node[vtx_mini]{}  -- (0.2,\x)node[vtx_mini]{} ;
			}
			
			
			\foreach \x in {3,2,1,0}{
			}

			\draw (2.25,-0.3) 
			+(0.2,0.2) node[vtx](x1){}
			+(-0.2,0.2) node[vtx](x2){}
			+(-0.2,-0.2) node[vtx](x3){}
			+(0.2,-0.2) node[vtx](x4){};
			\draw[yw]
			(x1) -- (x2) -- (x3) -- (x4) -- (x2) (x4)--(x1) -- (x3)
			;
			
			\draw[yw] (2,3*0.2+0.2) -- (2.5,2*0.2+0.2) to[bend left] (x4);;

			\foreach \x in {3,2,1,0}{
				\draw[yw] (2,\x*0.2+0.2) node[vtx_mini]{}-- (x1);
				\draw[yw] (2,\x*0.2+0.2)  node[vtx_mini]{}to[bend right] (x3);
			}
			
			\foreach \x in {2,1,0}{
				\draw[yw] (2,\x*0.2+0.2) node[vtx_mini]{} -- (2.5,\x*0.2+0.4) node[vtx_mini]{};
			}
			\draw[yw] (2.5,0.2) node[vtx_mini]{} to[bend left] (x4);

			\draw (2,-1.3) node{$W$};
			\draw (0,-1.3) node{$U$};
			\draw[blue] (2.8,0.2) node{$Y_W$};
			\draw[red] (0.7,-0.5) node{$Y_U$};
			\draw (1,-1.7) node{(a)};
			\end{tikzpicture}
			\hskip 3em		
			\begin{tikzpicture}[scale=1.5]
			\draw (0,0) ellipse (0.5cm and 0.8cm);
			\draw[rounded corners=20pt] (1,-1) rectangle (3,1.2) ;
			\draw (0,-0.4) node[vtx,label=left:$ $](v){};
			\foreach \x in {-2,-1}{
				\draw[zu] (v) -- (1.3,0.4*\x) node[vtx_mini]{} -- ++(0,0.2) node[vtx_mini]{}  -- (v);
			}
			\draw[zu] (v) -- (1.3,0.4*0) node[vtx_mini]{};
			
			\draw (0,0.4) node[vtx,label=left:$ $](u){};
			\foreach \x in {2,1}{
				\draw[zu] (u) -- (1.3,0.4*\x) node[vtx_mini]{} -- ++(0,0.2) node[vtx_mini]{} -- (u);
			}
			\draw[zu] (u) -- (1.3,0.4*0.5) node[vtx_mini]{};
			\draw[zu](u)--(v);
			
			\draw (1.3,0.4*0)  node[vtx_mini](x1){}  (1.6,0.1) node[vtx_mini](x2){} (1.3,0.4*0.5)  node[vtx_mini](x3){};
			\draw[green!30!black,line width=1pt] (x1)  -- (x2) -- (x3) ;
			
			\foreach \x in {3,2,1,0}{
				\draw[zw] (2,\x*0.2+0.2) node[vtx_mini]{}-- (2.5,\x*0.2+0.2) node[vtx_mini]{};
			}

			\draw (2.25,-0.3) node{$X$};
			\draw[zw] (2.25,-0.3) 
			+(0.2,0.2) node[vtx](x1){}
			+(-0.2,0.2) node[vtx](x2){}
			+(-0.2,-0.2) node[vtx](x3){}
			+(0.2,-0.2) node[vtx](x4){}
			(x1) -- ($(x1)!0.5!(x4) + (0.2,0)$) node[vtx_mini]{}-- (x4)
			(x2) -- ($(x2)!0.5!(x3) + (-0.2,0)$) node[vtx_mini]{}  -- (x3)
			;
			\draw (2,-1.3) node{$W$};
			\draw (0,-1.3) node{$U$};
			\draw (2.7,0.5) node{$M$};
			
			\draw[blue] (2.8,0) node{$Z_1$};
			\draw[red] (0.75,0.05) node{$Z_2$};
			\draw[green!30!black] (1.5,-0.1) node{$Z_3$};
			\draw (1,-1.7) node{(b)};			
			\end{tikzpicture}
		\end{center}
		\caption{
		(a) Missing edges in $Y_W$ are colored blue and edges in $Y_U$ are red. 
		(b)
			Edges in $Z_1$ are colored blue, edges in $Z_2$ are red and in $Z_3$ green.
			The same vertices are on the right, where dashed are some of the  missing edges.
			Note that this is a sketch and vertices in $W$ can incident to both blue and red (dashed) edges.
			}\label{fig:decom} 
	\end{figure}
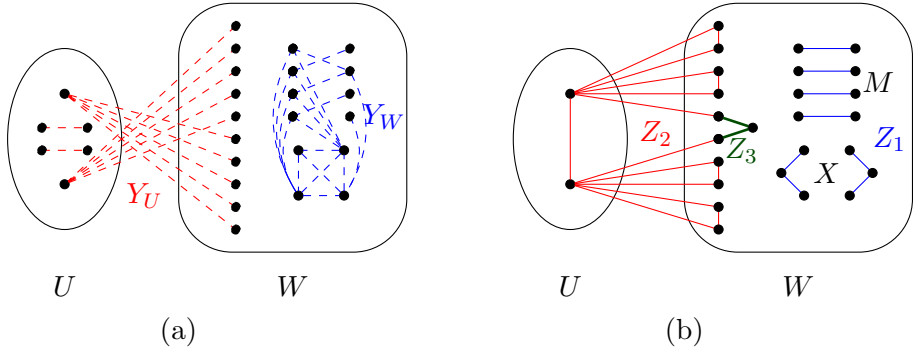

	\begin{claim*}
		The above steps can be carried out as stated. Moreover, the obtained decomposition $\C D$ of $G$ has at most $|M| + |X| + \binom{|U|}{2} + 2|U| + 6$ copies of~$K_2$.
	\end{claim*}
	
	\begin{proof}[Proof of Claim.]
In order to do Step $1$ as stated, we can iteratively pick any two new vertices $x,y\in X$ and then an arbitrary vertex $z$ which is suitable as the middle point for a cherry on~$xy$. Note that the number of choices for $z$ is at least $n-2-2cn$, the number of common neighbours of $x,y\in X\subseteq W$, minus $|X|-1$, the number of vertices previously used as middle points. This is positive by~\eqref{eq:|X|} and $c\ll 1$, so we can always proceed. Note for future reference that every vertex is incident to at most 3 edges removed in Step~1.
Also, Step $1$ adds $|Z_1|=|M| + 2(\lfloor |X|/2\rfloor) \le |M| + |X|$ copies of $K_2$ to $\C D$.

Clearly, Step~2 can always be processed.  Consider the moment when we apply Step $2$ to some $u \in U$. In the current graph, the induced subgraph $G[\Gamma(u) \cap W]$ has minimum degree at least $|\Gamma(u)\cap W|-cn-3$, which is at least $|\Gamma(u)\cap W|/2$ since
$|\Gamma(u)|\ge n/8-3$. So by Dirac's theorem, this subgraph has a  matching covering all but at most one vertex, that is,  all edges between $u$ and $W$ except at most one are decomposed as triangles in Step~2. 
Let $U'$ be the set of those $u\in U$ for which an exceptional edge occurs.
Thus we have $|U'|\le |U|$ copies of $K_2$ connecting $U$ to $W$ that are added to $\C D$ in Step~2.
There are trivially at most $\binom{|U|}{2}$ edges with both endpoints in $U$.
So Step $2$ adds $|Z_2|\le \binom{|U|}{2} + |U|$ copies of $K_2$ to $\C D$.
Note that all edges incident to $U$ are decomposed after Step~$2$.

Since all vertices of $W$ but at most one had even degrees before Step~2, we have that $S'$ has at most $|U'|+1\le |U|+1$ vertices.  Similarly as in Step~1, a simple greedy algorithm finds all cherries as stated Step 3(a). (Note that $S'$, as the set of all odd-degree vertices, has even size.)

The minimum degree of $G[W]$ after Step 3(a) is at least $0.99n$, since each $w \in W$ has at most $2|U|+6$ incident edges removed (at most $2|U|$ from Step $2$ and at most 3 in each of Steps 1 and~3(a)). Thus, we can find the required 4- or 5-cycle in Step 3(b).

Clearly, we add $|Z_3|\le |S'|+5\le |U|+6$ copies  of $K_2$ to $\C D$ in Step~3.

Note that, at the end of Step~3, the graph $G[W]$ has minimum degree at least, say,~$0.98n$ while all its degrees are even. By Theorem~\ref{thm:decomp}, 
all remaining edges can be decomposed using only triangles, so Step $4$ indeed removes all remaining edges.
		
Step $4$ adds no additional $K_2$'s, so the total number of $K_2$'s in $\C D$ is 
		\[
		|Z_1|+|Z_2|+|Z_3| \leq |M|+|X|+\binom{|U|}{2} + 2|U|+6,
		\]
 finishing the proof of the claim.	\end{proof}
	
	Now we compute the cost of $\C D$. Using the notation from above, we have
	\begin{align}
	w&\le\pi_3(G)-\binom{n}{2}\le -|Y_U|-|Y_W|+|Z_1|+|Z_2|+|Z_3|\nonumber\\
 &\le -|Y_U|-|Y_W|+|M|+|X|+\binom{|U|}{2} + 2|U|+6.\label{eq:pi3}
	\end{align}
	Substituting the bounds from \eqref{eq:yw} and \eqref{eq:yu} and rearranging the terms,
we get
	\begin{equation}\label{eq:expanded}
	w\le 
 \left(2\binom{|U|}{2}+2|U|-cn|U|+6\right) + \left(3-|X|\right)\left(\frac{|X|}{2}+|M|\right)+\left(\I 1_{|X|=1}-2\right)|M|.\end{equation}
	
First, suppose that $|U|>0$. Then, the estimate $|U|\le 2\delta n/c$ yields that
\[ 2\binom{|U|}{2}+2|U|-cn|U|+6 \le - cn|U|/2 \le -cn/2.\]
Since $w\ge 2$, we must have that $|X|\le 1$. 
Observe that $n$ is odd as otherwise $w\ge n/2$ and, by $|M|\le n/2$, the cases $|X|\in\{0,1\}$ also contradict~\eqref{eq:expanded}. So every vertex of degree $n-1$ has even degree, meaning that every vertex of $S$ is in some pair from $Y_W$ or~$Y_U$. Hence, $2|M|\le 2|Y_W|+|Y_U|$. Substituting this into the right-hand size of \eqref{eq:pi3} and using our bound on $|Y_U|$ from~\eqref{eq:yu}, we obtain 
		\begin{align*}w
&\le -\frac{|Y_U|}2+|X|+\binom{|U|}{2}+2|U|+6
		\le \frac{3}{2}\binom{|U|}{2}+2|U|-\frac{cn|U|}2+7,
		\end{align*}
		which again is negative for $|U|>0$ and large $n$, contradicting $w\ge 2$.

Thus $U$ is empty and, by the assumption of Case~1, $S$ is also empty (and so are $X$ and $M$).
This gives that the initial graph $G$ has
minimum degree at least $(1-c)n$,  $|Z_1|=|Z_2|=0$, $S'=\emptyset$, and no $K_2$'s are added to $\C D$ in Step~3(a).

If $n$ is even, then every vertex of $G$ has at least one missing edge, $e(G)\le {n\choose 2}-\frac n2$ and
 $$
\pi_3(G)\le {n\choose 2}-\frac n2+|Z_3|\le {n\choose 2}-\frac n2+5,
$$
 which is strictly less than $\pi_3(K_n)$, a contradiction.

Let $n$ be odd and let $r:={n\choose 2}-e(G)$ be the number of missing edges in~$G$. Suppose that $r>0$, as otherwise $G=K_n$ and we are done. The upper bound on $\pi_3(G)$ given by $\C D$ is $\rho_r+{n\choose 2}-r$, where we define $\rho_r$ as the unique element of $\{0,4,5\}$ with ${n\choose 2}-\rho_r-r\equiv 0\pmod 3$.
Therefore, $r \le 3$ as otherwise $\pi_3(G) \le {n\choose2}+1$ contradicting $w \ge 2$. On the other hand, all the degrees of $\O G$ are even so $r=3$ and the only non-empty component of $\O G$ is a triangle.
However, this contradicts $w\ge2$ because
\[
\pi_3(G) = \left\{\begin{array}{ll}
{n \choose 2} - 1,&n\equiv 1,3\pmod 6,\\
{n \choose 2} + 1,& n\equiv 5 \pmod 6.\end{array}\right.
\]

\medskip \noindent\textbf{Case 2:} $U=\emptyset$ and $S\not=\emptyset$.\medskip

Some things simplify in this case (as we do not need to deal with $U$). On the other hand, we have to be a bit more careful with calculations, as the new extremal graphs ($K_n$ minus a matching) fall into this case. In particular, removing a 4- or 5-cycle may be too wasteful here. So we construct a decomposition $\C D$ of $G$ as follows. Recall that $M$ is a maximum matching in $G[S]$ and $X$ is  the set of vertices of $S$ not matched by~$M$. 

	\begin{description}
\item[Step 1:] Make the graph triangle-disivible by removing the following as $K_2$'s. If $X=\emptyset$, then remove all but one edge $xy\in M$ and a path of length $\rho+1\in\{1,2,3\}$ whose endpoints are $x$ and $y$ (thus, for $\rho=0$, we remove just the matching~$M$). If $X$ is non-empty, then remove $M$ and the edge sets of some $|X|/2-1$ paths of length 2 and  one path of length $\rho+2\in \{2,3,4\}$ so that their degree-1 vertices partition $X$ and their degree-2 vertices are pairwise distinct. 

\item[Step 2:] Decompose the rest perfectly into triangles.
\end{description}

Note that $S$, the set of all odd-degree vertices of $G$, has even size (and also $|X|=|S|-2|M|$ is even).
Since the minimal degree of $G$ is at least~$(1-c)n$, a  simple greedy algorithm achieves Step~1 (and Theorem~\ref{thm:decomp} takes care of Step~2). 

The decomposition $\C D$ has exactly $|M|+|X|+\rho$ copies of $K_2$.
Also, $e(G)= {n\choose 2}-|Y_W|$. Thus
\begin{equation}\label{eq:UEmpty}
w\le \pi_3(G)- {n\choose 2}\le -|Y_W|+|M|+|X|+\rho.
\end{equation}

Using~\eqref{eq:yw} and that $|X|\not=1$ (since $|X|$ is even),
we obtain that
\begin{equation}\label{eq:uw}
w \le \left(3-|X|\right)\left(\frac{|X|}{2}+|M|\right)-2|M|+\rho.
\end{equation}

Moreover, $|X|\le 2$ as otherwise $2\le w\le \rho-2-3|M|$ contradicting $\rho\le 2$. Thus $X$ has either 0 or 2 elements.

Suppose that $X=\emptyset$. First, let $n$ be even. Then every vertex not in $S$ is incident to at least one non-edge of $G$, $|Y_W|\ge (n-2|M|)/2$, and by~\eqref{eq:UEmpty}, \[n/2\le w\le 2|M|+\rho-n/2.\]
If $2|M|\le n-2$, then all inequalities here become equalities and thus $|M|=\frac{n-2}2$, $|Y_W|=1$, $\rho=2$, $w=\frac n2$, and $n\equiv 0,2\pmod 6$. However, then the graph after Step 1 has exactly ${n\choose 2} -1- \frac{n-2}2-2$ edges, which is not divisible by~3, a contradiction. Thus $2|M|=n$, the copies of $K_2$ in the decomposition contains a perfect matching of $G$, and $\pi_3(G)\le \pi_3(K_n)$ with equality only if $G=K_n$, giving the desired. So suppose that $n$ is odd. Since every vertex of $S$ has to be incident to a missing edge of $G$, we have $|Y_W|\ge |S|/2=|M|$ and the bound in~\eqref{eq:UEmpty} becomes $w\le \rho$. It follows that we have equality throughout, $|Y_W|=|M|$, $w=\rho=2$, $n\equiv 1,3\pmod 6$, and ${n\choose 2}-|M|-\rho\equiv 0\pmod 3$; the last gives that $|M|\equiv2\pmod 3$. Thus $G$ is as required.

Finally, it remains to consider the case when $|X|=2$. This time,~\eqref{eq:uw} yields that \[2\le w\le \rho-|M|+1 \le 3.\]
Therefore,  $|M|\le 1$, and $n \equiv 1,3 \pmod 6$ as otherwise $w\ge 4$. 
If $|M|=1$, then we have equality everywhere, giving that $w=\rho=2$, $|S|=4$ and $|Y_W|=3$. However, then the graph after Step~1 has ${n\choose 2}-|Y_W|-|M|-|X|-\rho={n\choose 2}-8$ edges, which is not divisible by 3, a contradiction. Thus $M$ is empty, $\rho\in\{1,2\}$ and $S=X$. By~\eqref{eq:UEmpty}, $|Y_W|\le 2$ and hence $|Y_W|=1$.
In other words, $G=K_n^-$. However, then the graph after Step~1 has ${n\choose 2}-1-(2+\rho)$ edges, which is not divisible by~3. (Alternatively, Theorem~\ref{thm:decomp} gives that $\pi_3(K_n^-)-{n\choose 2}<2=w$.) This contradiction finishes Case~2 and the proof of the lemma.\end{proof}

 \begin{lemma}\label{lm:Kn}
There exist constants $\delta>0$ and $n_1\in \mathbb{N}$ such that the following holds.
Let $G$ be a graph  on $n \geq n_1$ vertices maximizing  $\pi_3(G)$ among all graphs 
that are $\delta n^2$-close to $K_n$.
Then $G\in\C E'_n$.
\end{lemma}
\begin{proof}
Let $n_0$ and $\delta$ be the constants from Lemma~\ref{lm:KnDeg}. 
We claim that, for example, $n_1:=2n_0$ is enough for the conclusion of Lemma~\ref{lm:Kn} to hold. 
Indeed, take any extremal graph $G$ of order $n\ge n_1$. 
If $G$ satisfies the assumption on minimum degree of Lemma~\ref{lm:KnDeg}, then we are done.
Hence assume that the minimum degree of $G$ is less than $n/8$.
Let $G_n:=G$, and iteratively define a sequence of graphs $G_{n-1},G_{n-2},\ldots$ as follows. 
Given a graph $G_i$ of order $i$, if it has a vertex $x$ of degree less than $i/8$, let $G_{i-1}:=G_i-x$ be obtained from $G_i$ by removing the vertex $x$; otherwise stop. 
Note that the process does not reach $i<n/2$ for otherwise $G$ has roughly at least $(n/2)\times (n/4)$ non-edges, which is a contradiction to $G$ being $\delta n^2$-close to~$K_n$.

 Let $G_s$ with $|G_s|=s\ge n/2\ge n_0$ be the graph for which the above process terminates. 
 By Lemma~\ref{lm:KnDeg}, we have that $\pi_3(G_s)\le \frac{s^2}{2}+1$. By decomposing all edges in $E(G)\setminus E(G_s)$ as $K_2$'s, we obtain that
 $$
 \pi_3(G_n)\le \pi_3(G_s) + 2(n-s)\cdot\frac{n}8 \le \frac{s^2}{2}+1 +  (n-s)\cdot \frac n4.
 $$
 This is a convex function in $s$ so it is maximized on the boundary of $\frac{n}{2} \leq s \leq n-1$.
 If  $s = n/2$, we get   $\pi_3(G_n) \le n^2/4+2 < \binom{n}{2} \leq \pi_3(K_n)$.
 If $s = n-1$, we get
 \[
  \pi_3(G_n)\le \pi_3(G_s) + 2(n-s)\cdot\frac n8 \le \frac{(n-1)^2}{2}+1 +  \frac n4 \leq \binom{n}{2} - \frac n4+2 < \pi_3(K_n).
\]
In both cases, we get a contradiction to $G_n$ being extremal.
\end{proof}

\begin{proof}[Proof of Theorem~\ref{th:1}.] 

Choose sufficiently small constants in this order $1\gg \delta\gg 1/n_0>0$. 
In particular, $n_0$ is sufficiently large to satisfy Corollary~\ref{co:stab} for this $\delta$
as well as Lemmas~\ref{lm:keyT2n} and~\ref{lm:Kn}.
Let $G$ be an arbitrary graph of order~$n\ge n_0$ with $\pi_3(G)\ge \ell(n)$.
By Corollary~\ref{co:stab}, $G$ is $\delta n^2$-close to either $T_2(n)$ or $K_n$. 

If $G$ is close to  $T_2(n)$ then it must be $T_2(n)$ by Lemma~\ref{lm:keyT2n}.
If $G$ is close to  $K_n$ then it must be in $\C E_n'$ by Lemma~\ref{lm:Kn}.
By comparing the costs of optimal decompositions, we conclude that $G\in \C E_n$.
\end{proof}

\section{Extension to an arbitrary cost $\alpha$}\label{sec:alpha}

The goal of this section is to prove Theorem~\ref{thm:alpha}. Everywhere in this section, let~$n$ be sufficiently large.

First, note that the case $\alpha \geq 6$ is trivial. Indeed, the cost of a triangle is not better than a cost of three edges. Thus for every graph $G$ an optimal decomposition is to decompose all edges of $G$ as $K_2$'s. The unique graph maximizing the number of edges is $K_n$, so it is also the unique maximizer of $\pi_3^\alpha$ for every $\alpha \geq 6$.

Next, let us make some easy general observations which apply when $\alpha<6$. 
First, 
$$
\pi_3^\alpha(G)=\alpha \nu(G)+2(e(G)-3\nu(G))=2e(G)-(6-\alpha)\nu(G),
$$ 
 where
$\nu(G)$ denotes the maximum number of edge-disjoint triangles contained in $G$.
Also, if $\alpha_1 \le \alpha_2<6$, $\nu(G_1) \geq \nu(G_2)$ and   
$\pi_3^{\alpha_1}(G_1) > \pi_3^{\alpha_1}(G_2)$ for some graphs $G_1$ and $G_2$, 
then 
 \begin{equation}\label{eq:mon}
\pi_3^{\alpha_2}(G_1) - \pi_3^{\alpha_2}(G_2)=\pi_3^{\alpha_1}(G_1) - \pi_3^{\alpha_1}(G_2)+(\alpha_2-\alpha_1)(\nu(G_1)-\nu(G_2))>0.
\end{equation}
In particular, if $K_n$ is the maximizer of $\pi_3^{\alpha_1}$, it is also a maximizer for $\pi_3^{\alpha_2}$.

\subsection{The case $\alpha < 3$}

Next, we discuss the case $\alpha < 3$. Let $n$ be large and let $G$ be a $\pi_3^\alpha(n)$-extremal graphs.
Since $$
\pi_3^3(G)\ge \pi_3^\alpha(G)\ge \pi_3^\alpha(T_2(n))=\pi_3^3(T_2(n))=(1/2+o(1))n^2,
$$ 
 Corollary~\ref{co:stab} gives that $G$ is $o(n^2)$-close to $K_n$ or $T_2(n)$. Since $\alpha<3$, we have that $\pi_3^\alpha(T_2(n))\ge (1+\Omega(1))\pi_3^\alpha(K_n)$ and thus $G$ is close to $T_2(n)$. Now, Lemma~\ref{lm:keyT2n}
implies that $\pi_3^\alpha(G)\le \pi_3^3(G)\le \pi_3^3(T_2(n))=\pi_3^\alpha(T_2(n))$, with equality if and only if $G=T_2(n)$, giving the desired.

\subsection{The case $3 < \alpha < 4$}
This subsection proves  Theorem~\ref{thm:alpha} in case $3 < \alpha < 4$.

First, let us show that every $\pi_3^\alpha$-maximiser $G$ is in $K_n$ or $K_n^=$. 
Suppose for a contradiction that $G$ violates this.  In particular, we have
$\pi_3^{\alpha}(G) \geq \pi_3^{\alpha}(K_n)$.
By~\eqref{eq:mon}, we have that $\pi_3^3(G) \geq \pi_3^3(K_n)$.
For $n\to\infty$, it holds by Table~\ref{ta:1} that  $\pi_3^{\alpha}(K_n) \ge (1+\Omega(1))\,\pi_3^{\alpha}(T_2(n))$.
Hence $G$ needs to be close to $K_n$ and Lemma~\ref{lm:Kn} applies to $G$. 
In particular, this means that $n \equiv 1,3 \pmod 6$.
Lemma~\ref{lm:Kn} gives that all $\pi_3^3$-extremal graphs are obtained from $K_n$ by removing a matching of size congruent to $2$ modulo $3$. It follows from~\eqref{eq:mon} that, among these graphs, 
$\pi_3^{\alpha}$ is strictly maximized by $K_n^=$ since this graph has the largest~$\nu$.  

Theorem~\ref{thm:decomp} gives that
$3\nu(K_n^=)={n\choose 2}-6$.
Since $\pi_3^\alpha(G) \geq \pi_3^\alpha(K_n^=)$ and $\pi_3^3(G) < \pi_3^3(K_n^=)$, this implies by~\eqref{eq:mon} that $\nu(G) >  \nu(K_n^=)$.  Since also $\nu(G)<\nu(K_n)$ (otherwise $\pi_3^\alpha(G)< \pi_3^\alpha(K_n)$), we conclude that 
$3\nu(G)={n\choose 2}-3$, that is, exactly three pairs of vertices of $G$ are not included into some triangle from an optimal decomposition of~$G$. 
This implies that $G$ is a complete graph without one edge, or a path on three vertices, or a triangle.
Among these three candidates (that have the same $\nu$), $K^-$ has the largest size and thus maximizes $\pi_3^\alpha$. So $K^-$ is the only possible candidate for~$G$.
However, $\pi_3^\alpha(K_n^=) > \pi_3^\alpha(K_n^-)$ if $\alpha < 4$. This  contradiction finishes the proof in case $3 < \alpha < 4$.

Thus, every $\pi_3^\alpha$-maximiser is in $\{K_n,K_n^=\}$. It remains to compare these two graphs. Calculations based on Theorem~\ref{thm:decomp} show that
 $$
 \frac{\pi_3^\alpha(K_n^=)-\pi_3^\alpha(K_n)+4}{6-\alpha}=\nu(K_n)-\nu(K_n^=)=\left\{\begin{array}{ll} 0,& n\equiv 0,2,4,5\pmod 6,\\
 2,& n\equiv 1,3\pmod 6.
\end{array}\right.
$$
 Thus $\pi_3^\alpha(K_n)>\pi_3^\alpha(K_n^=)$ if $n\equiv 0,2,4,5\pmod 6$ and $\pi_3^\alpha(K_n^=)>\pi_3^\alpha(K_n)$ otherwise, as required.


\subsection{The case $4 \leq \alpha < 6$}

In this case we provide a direct proof, without using flag algebras or fractional decompositions. 
Let $n$ be large and let $G$ be any graph of order $n$ such that $\pi_3^\alpha(G) = \pi_3^\alpha(n)$.
Let $\C D$ be a decomposition of $G$ with minimum weight
consisting of $t$ triangles and $\ell$  edges.

If $G$ is a complete graph, then we are done. 
Hence we assume  there exists some pair of vertices $x,y \in G$ such that $xy \notin E(G)$. 
Let $G'$ be obtained from $G$ by adding the edge  $xy$. 
Let $\C D'$ be an optimal decomposition of $G'$ containing $t'$ triangles and $\ell'$ edges. 
Recall that finding an optimal decomposition is equivalent to maximizing a triangle packing, that is, $t'=\nu(G')$. Hence $t' \geq t$.

If  $xy$ is used as an edge in $\C D'$, then removing $xy$ from $\C D'$ gives a decomposition of $G$ with cost $\pi_3^\alpha(G')-2$, contradicting the maximality of $G$. 
Therefore $xy$ must appear in a triangle $xyz \in \C D'$. We now construct a decomposition $\C D^*$ of $G$ by removing $xyz$ from $\C D'$ and adding the edges $xz$ and $yz$. 
Since the total cost of $\C D^*$ is $\alpha(t'-1) + 2(\ell'+2)$ we have
 \[
 \pi_3^\alpha(G) \leq \mbox{cost}(\C D^*) = \alpha(t'-1) + 2 (\ell' +2) = \alpha t'  + 2 \ell' - \alpha  + 4  \leq \alpha t' + 2 \ell' = \pi_3^\alpha(G'),
 \] 

which contradicts the maximality of $\pi_3^\alpha(G)$ if at least one of the inequalities is strict. 
Hence $\alpha = 4$, 
$xy$ must be in a triangle in $\C D'$ and $\pi_3^\alpha(G') = \pi_3^\alpha(n)$. 

This means that it is possible to keep adding edges to $G$, which results in a sequence of graphs $G,G',\ldots,K_n$ where an optimal decomposition of each of these graphs has cost $\pi_3^\alpha(n)$, i.e.\ they all are $\pi_3^\alpha$-extremal graphs.
Note that we can add missing edges to $G$ in any order, always obtaining a sequence of extremal graphs.

This allows us to reverse the process and examine a sequence of edge removals from $K_n$.

%
%
%
%
%
%
%
%



Suppose that $G$ is obtained from $K_n$ by removing the edge $xy$, i.e.\ $G'$ is $K_n$.  
Notice that if $\ell' > 0$, i.e.\ the optimal  decomposition of $K_n$ contains an edge, then there exist an option for $\C D'$ that  contains the edge $xy$, which was already ruled out.
This means that $K_n$ is triangle-divisible, which is
the case if and only if $n \equiv 1,3 \pmod{6}$.

Now assume that $G$ is missing more than one edge. 
Hence $K_n^-$ must be also extremal.  By above, $n \equiv 1,3 \pmod{6}$, $K_n$ is triangle-divisible, and $\pi_3^4(n) = 4\nu(K_n)$, where $\nu(K_n)=\frac13{n\choose 2}$.

Suppose that $G$ is obtained from $K_n$ by removing two edges $uv$ and $xy$.
First, suppose that $u = x$. Let $\C D^\star$ be a decomposition of $G$ into triangles and one edge $vy$. This gives
\[
\pi_3^4(G) \leq  \mbox{cost}(\C D^\star) =  4(\nu(K_n)-1) + 2  < 4\nu(K_n) =  \pi_3^4(n),
\]
contradicting the maximality of $\pi_3^4(G) $.
Hence $xy$ and $uv$ form a matching.
Notice that $x$, $y$, $u$, and $v$ have odd degrees in $G$, so $\ell \geq 2$ for else we are unable to fix the parity of the vertices $x$, $y$, $u$, and $v$.
Now $\binom{n}{2} - \ell-2$ needs to be divisible by 3, so $\ell \geq 4$. 
There indeed exists a decomposition with $\ell = 4$  by taking edges $xu$, $xv$, $yu$, and $yv$ and rest as triangles.
This gives
\[
\pi_3^4(G) = 4(\nu(K_n) - 2) +  2 \cdot 4 = \pi_3^4(n).
\]
Therefore, $G$ is extremal.

Suppose that $G$ is obtained from $K_n$ by removing three edges $uv$, $xy$, and $zw$.
Since $G'$ must be $K_n$ without a matching, $uv$, $xy$, and $zw$ also form a matching.
Let $\C D^\star$ be a decomposition of $G$ into triangles and edges
$ux$, $yz$, and $vw$. This gives
\[
\pi_3^4(G) \leq  \mbox{cost}(\C D^\star) = 4(\nu(K_n)-2) + 6  < 4\nu(K_n) =  \pi_3^4(n),
\]
contradicting the maximality of $\pi_3^4(G)$.
This implies that $G$ cannot be obtained from $K_n$ by deleting three or more edges, thus finishing the proof of this case and of Theorem~\ref{thm:alpha}.
%
%
%
%
%

\section{Related results}\label{sec:related}

A related question of Erd\H os (see e.g.,~\cite{Erdos71cma}) asks for the largest $t=t(n,m)$ such that every graph with $n$ vertices and $t_2(n)+m$ edges has at least $t$ edge-disjoint triangles. Of course, $t\le m$. Gy{\H o}ri~\cite{Gyori88} (see~\cite{Gyori92} for a correction) showed, for large $n$, that $t\ge m-O(m^2/n^2)$ if $m=o(n^2)$, and $t=m$ if $n$ is odd and $m\le 2n-10$ or $n$ is even and $m\le 3n/2-5$. Moreover, the last two bounds on $m$ are sharp.

More recently, Gy{\H o}ri and Keszegh~\cite{GyoriKeszegh17c} proved that every $K_4$-free graph with $t_2(n)+m$ edges has $m$ edge-disjoint triangles.

Theorem~\ref{th:1} shows that the maximum of $\pi_3(G)$ is attained for $G=T_2(n)$ or $G=K_n$. However, if we restrict the set of graphs under consideration to graphs of a particular edge density, the decomposition is perhaps cheaper. Note that if the optimal decomposition of a graph $G$ contains $t$ triangles and $\ell$ edges, then $\pi_3(G)=
2e(G)-3t$. That is, we have that $\pi_3(G)=2e(G)-3\nu(G)$, where as before $\nu(G)$ denotes the maximum number of edge-disjoint triangles in~$G$.
Then Theorem \ref{thm:klmp} implies an inequality between the edge density of $G$ and its \emph{triangle packing density} which we denote by $\nu_d(G):=3\nu(G)/\binom{n}{2}$:

\begin{corollary}[of Theorem \ref{thm:klmp}]\label{cor:end} Let $G$ be a graph with $d\binom{n}{2}$ edges. Then
\[\nu_d(G)\ge 2d-1+o(1).\]
\end{corollary}

We also have that $\nu_d(G)\le d$, which is tight for all graphs which are the union of edge-disjoint triangles.

A question reminiscent of the seminal result of Razborov on the minimal triangle density in graphs \cite{Razborov08} (see also \cites{PikhurkoRazborov17,LiuPikhurkoStaden20}) would be to determine the exact lower bound on $\nu_d(G)$ in terms of $d$ (answering asymptotically the question of Erd\H os stated above).

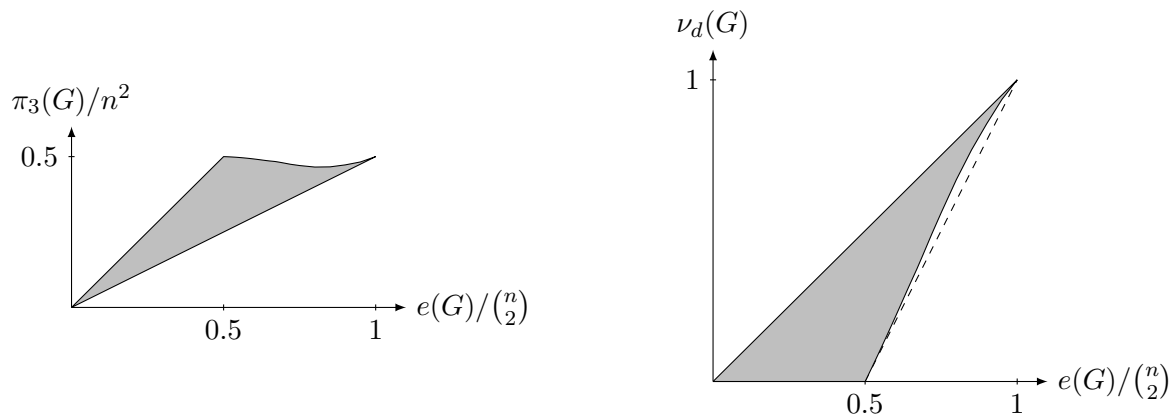
\begin{figure}[H]

\begin{subfigure}{0.45\textwidth}

\begin{tikzpicture}[scale=4]
\draw[-latex] (0,0) -- (1.1,0) node[right]{$e(G)/\binom{n}{2}$};
\draw[-latex] (0,0) -- (0,0.6) node[above]{$\pi_3(G)/n^2$};;
\draw (0.5,0) +(0,0.01)-- +(0,-0.01) node[below]{0.5};
\draw (1,0) +(0,0.01)-- +(0,-0.01) node[below]{1};
\draw (0,0.5) +(0.01,0)-- +(-0.01,0) node[left]{0.5};

 \draw[fill=gray!50](0,0) --
(0.500000000000000,0.5)--
(0.510000000000000,0.49963632)--
(0.520000000000000,0.49917183)--
(0.530000000000000,0.49860894)--
(0.540000000000000,0.4979502)--
(0.550000000000000,0.49719824)--
(0.560000000000000,0.49635583)--
(0.570000000000000,0.49542582)--
(0.580000000000000,0.4944112)--
(0.590000000000000,0.49331507)--
(0.600000000000000,0.49214063)--
(0.610000000000000,0.49089119)--
(0.620000000000000,0.4895702)--
(0.630000000000000,0.48820527)--
(0.640000000000000,0.48691476)--
(0.650000000000000,0.48571598)--
(0.660000000000000,0.48470355)--
(0.670000000000000,0.48369721)--
(0.680000000000000,0.4820356)--
(0.690000000000000,0.48025622)--
(0.700000000000000,0.47838868)--
(0.710000000000000,0.47648692)--
(0.720000000000000,0.47459943)--
(0.750000000000000,0.47035407)--
(0.800000000000000,0.46538961)--
(0.850000000000000,0.4660294)--
(0.900000000000000,0.47258249)--
(0.950000000000000,0.48229491)--
(1.0,0.5)--
(0,0);

\end{tikzpicture}
\end{subfigure}
\hfill
\begin{subfigure}{0.45\textwidth}
\begin{tikzpicture}[scale=4]
\draw[-latex] (0,0) -- (1.1,0) node[right]{$e(G)/\binom{n}{2}$};
\draw[-latex] (0,0) -- (0,1.1) node[above]{$\nu_d(G)$};
\draw[dashed] (0.5,0) -- (1,1);
\draw (0.5,0) +(0,0.01)-- +(0,-0.01) node[below]{0.5};
\draw (1,0) +(0,0.01)-- +(0,-0.01) node[below]{1};
\draw (0,1) +(0.01,0)-- +(-0.01,0) node[left]{1};

 \draw[fill=gray!50](0,0) --
(0.500000000000000,0)--
(0.510000000000000,0.0207274)--
(0.520000000000000,0.0416563)--
(0.530000000000000,0.0627821)--
(0.540000000000000,0.0840996)--
(0.550000000000000,0.105604)--
(0.560000000000000,0.127288)--
(0.570000000000000,0.149148)--
(0.580000000000000,0.171178)--
(0.590000000000000,0.19337)--
(0.600000000000000,0.215719)--
(0.610000000000000,0.238218)--
(0.620000000000000,0.26086)--
(0.630000000000000,0.283589)--
(0.640000000000000,0.30617)--
(0.650000000000000,0.328568)--
(0.660000000000000,0.350593)--
(0.670000000000000,0.372606)--
(0.680000000000000,0.3959929)--
(0.690000000000000,0.419488)--
(0.700000000000000,0.443223)--
(0.710000000000000,0.467026)--
(0.720000000000000,0.490801)--
(0.750000000000000,0.559292)--
(0.800000000000000,0.669221)--
(0.850000000000000,0.767941)--
(0.900000000000000,0.854835)--
(0.950000000000000,0.93541)--
(1.0,1.0)--
(0,0);

\end{tikzpicture}
\end{subfigure}
\caption{Asymptotic bounds on possible values of $\pi_3(G)$ and $\nu_d(G)$. The dashed line is simply $y=2x-1$ for a better display of the shape.}\label{fig:asinew}
\end{figure}

Some flag algebra computations yield numerical asymptotic lower bounds on $\nu_d(G)$ with different edge densities between 0.5 and 1. The result, depicted in Figure~\ref{fig:asinew}, suggests that the true asymptotic shape of the region $\{(d,\nu_d(G)): 0\le d\le 1, G \mbox{ graph}\}$ may indeed have a richer structure.

\section{Acknowledgement}
Work on this project was started at Rocky Mountain Great Plains Graduate Research Workshops in Combinatorics 2018. 
The work is partially supported by NSF-DMS grants \#1603823 and \#1604458
"Collaborative Research: Rocky Mountain Great Plains Graduate Research Workshops in Combinatorics" and by NSA grant \#H98230-18-1-0017, "The 2018 and 2019 Rocky Mountain – Great Plains Graduate Research Workshops in Combinatorics." 
We would like to thank Ryan R. Martin for fruitful discussions during the early stages of this project, and Ben Barber for suggesting the problem presented in Section \ref{sec:related}.

\bibliography{DecompositionsIntoK2K3}


\end{document}

\section*{Appendix}\label{sec:appendix}


Since it is not easy to find a copy of the paper~\cite{Gyori88}, we present a self-contained proof of Case 1 of Lemma~\ref{lm:keyT2n}. \comment{Y: The Gy\H ori paper is also published in Combinatorica \cite{Gyori91}, which is on SpringerLink. If we want to have this proof nonetheless (say, for completeness), then it could go to an appendix.}

\begin{lemma}\label{lem:closetobipartite}
There exists $\delta > 0$ and $n_0$ such that every graph $G$ on
 $n\ge n_0$ vertices which is $\delta n^2$-close to $T_2(n)$ in edit distance satisfies $\pi_3(G)\le \pi_3(T_2(n))$ with equality if and only if $G\cong T_2(n)$.
\end{lemma}
\begin{proof}
Pick a \emph{max-cut partition} $V(G)=V_1\cup V_2$, that is, one that maximizes the number of edges across the parts. Since
$e(G)\ge \pi_3(G)/2\ge t_2(n)$ and $G$ is $\delta n^2$-close to $T_2(n)$, we have that $e(G[V_1,V_2])\ge t_2(n)-\delta n^2$, where $t_2(n):=e(T_2(n))=\floor{n^2/4}$. In particular, we have that
\beq{eq:vi}
||V_1|-n/2|\le 2\sqrt\delta n \mbox{~ and ~}||V_2|-n/2|\le 2\sqrt\delta n.
\eeq
Let $B=E(G[V_1])\cup E(G[V_2])$ denote the set of all edges within $V_1$ and $V_2$, which we will call \emph{bad}, and let $M=\overline{E}[V_1,V_2]$ be the set of pairs of vertices in $V_1\times V_2$ which are non-edges in $G$ (we call those \emph{missing pairs}). Denote $b:=|B|$ and $m:=|M|$.

For $i\in\{1,2\}$, let $U_i:=\{x\in V_i\mid d_B(x)\ge cn\}$ and $U:=U_1\cup U_2$, that is, $U$ consists of vertices with "large" bad degree. Since $|G\bigtriangleup K[V_1,V_2]|\le \delta n^2$, we have $|U|\cdot cn \le 2\delta n^2$ and thus 
\beq{eq:u}
u:=|U|\le (2\delta/c) n.
\eeq

We now proceed with decomposing the graph $G$ into triangles and edges in three steps:
\begin{description}
\item[step 1.] While $G$ contains a triangle $xyz$ with $x\in U$, $xy\in B$ and $xz,yz\not \in B$, remove $xyz$ from $G$. Denote the resulting graph by $G'$. \comment{Y: Why is it important that $xy$ is the only bad edge in the triangle?}
\item[step 2.] While $G'$ contains a triangle $xyz$ with $xy\in B$, $x,y\not \in U$, remove $xyz$ from $G'$ in a ``uniform'' way.
\item[step 3.] Decompose the rest of the graph into copies of $K_2$. 
\end{description}

The following claim states that step~1 removes all but at most $\rho:=2\sqrt{\delta/c} n$ bad edges at any vertex of $U$.

\begin{claim}\label{cl:1} $|\Gamma_{G'}(x_1)\cap V_1|\le \rho$ for all $x_1\in U_1$ and $|\Gamma_{G'}(x_2)\cap V_2|\le \rho$ for all $x_2\in U_2$.\end{claim}

\begin{proof}[Proof of claim.] It suffices to show the claim holds for vertices $x\in U_1$, as the other case follows by symmetry. Let $X_i:=\Gamma_G(x)\cap V_i$ for $i\in\{1,2\}$ and $X_i':=\Gamma_{G'}(x)\cap V_i$. Note that by the max-cut property we have $|X_1|\leq |X_2|$. We now show that this inequality holds approximately for $X_1'$ and $X_2'$ as well. 

Indeed, the total number of $G'$-non-edges between $X_1$ and $X_2$ is at most $\delta n^2$ (non-edges of $G$) plus $nu$ (as we remove trivially at most $n/2$ triangles containing any vertex of $U$ with each triangle having at most two $[X_1,X_2]$-edges). The only possibility that a removed triangle $xyz$ has more $[x,X_2]$-edges than $[x,X_1]$-edges is when the edge $yz$ has at least one vertex in $U$. Thus, there are at most $u$ such triangles and we have
$|X_2'|\ge |X_1'|-2u$. 

Now, we observe that $G'$ cannot have any edges between $X_1'$ and $X_2'$ for otherwise we can remove further triangles from $G'$. Putting the above observations together, we obtain that
$$
\delta n^2+nu\ge |X_1'|\cdot |X_2'|\ge
|X_1'|\cdot (X_1'-2u).
$$ 
Combined with \eqref{eq:u}, this inequality gives $|X_1'|\le \rho$, as required.\end{proof}
 
Let us observe one useful property of $G'$.  Consider any  $x\in V_1\setminus U_1$. Since $x$ is incident with fewer than $cn$ bad edges and has $G$-degree at least $(\ell(n)-\ell(n-1))/2\ge n/2-1$, we have by~\eqref{eq:vi} that $d_M(x)\le v_2- (n-2)/2+cn\le 2cn$. The number of $[x,V_2]$-edges removed at step 1 is at most $2|U|$, since for each triangle removed in Step 1, we remove exactly one $[x,U]$-edge. 

Thus the number of
$G'$-non-neighbours of $x$ in $V_2$ is at most 
\beq{eq:x}
|V_2\setminus \Gamma_{G'}(x)|\le 2cn+2u.
\eeq
By a symmetric argument the same holds for any $x\in V_2\setminus U_2$.

Let us describe Step~2, in which we remove some further triangles from $G'$. 
We construct an auxiliary digraph $D$ on $V(G)$ in the following way. First, initialise $D$ as the empty graph. Then, repeat the following as long as possible: pick a bad edge $xy$ disjoint from $U$, without loss of generality $xy\subseteq V_1\setminus U$, among all current common neighbours $z\in V_2$ of $x$ and $y$ pick one with the smallest in-degree in $D$, remove the triangle $xyz$ from the current graph, and add the arcs $(x,z)$ and $(y,z)$ to $D$. 

\begin{claim}\label{cl:indegree} The in-degree in $D$ of every vertex in $V(G)\setminus U$
is at most $\beta:=2\delta n$ during any point of Step~2. Also, Step~2 removes all bad edges disjoint from $U$.\end{claim}

\begin{proof}[Proof of claim.] Suppose on the contrary that the first part of the claim is false and consider the first time that this happens, without loss of generality when a vertex $z\in V_1$ gets in-degree in $D$ at least $\beta$ because of the removed triangle on $xyz$, where  $xy\subseteq V_1\setminus U$. The out-degree of $x$ in $D$ is at most $d_B(x)\le cn$. The in-degree in $D$ of $x$ is less than $\beta$. The same applies to $y$. Thus by~\eqref{eq:x}, when we remove the triangle on $xyz$, the vertices $x,y$ have between them
at most $2(2cn+2u+cn+\beta)$ non-neighbours in $V_1$. In other words, the number of common neighbours of $x,y$ in $V_1$ is at the current moment at least
\beq{eq:common}
|V_1| - 2(2cn+2u+cn+\beta)\ge (1/2-7c)n,
\eeq 
since by \eqref{eq:vi} we have $|V_1-n/2|\le 2\sqrt{\delta}n$. Each of these common neighbours has in-degree in $D$ at least $\beta-2$ because the in-degree of the chosen common neighbour $z$ is at least this amount.
 Thus the number of arcs in $D$ is at least $2(\beta-2)(1/2-7c)n$. On the other hand, it is at most $2b\le \delta n^2$, which contradicts our choice of~$\beta$ and proves the first part.
 
The second part now easily follows. By~\eqref{eq:x} and (the proof of) the first part, the number of common neighbours in $V_{3-i}$ of any  bad edge $xy \subseteq V_i\setminus U$ at any moment of Step~2 is lower bounded by the expression in~\eqref{eq:common} and this bound is strictly positive. Thus there is always at least one choice to remove a triangle containing $xy$.\end{proof} 

Now it remains to estimate the value of the obtained decomposition. For $i\in\{1,2,3\}$, let $b_i$ be the number of bad edges removed at Step $i$.
By the second part of Claim \ref{cl:indegree}, $b_3$ counts only bad edges intersecting $U$. By Claim~\ref{cl:1}, for every vertex of $U$, the ratio of the number of incident $b_3$-edges to that for $b_1$-edges is at most
$\rho/(cn-\rho)$. Since we may double count some edges, we conclude that
$b_3\le \frac{2\rho}{c n -\rho}\, b_1$. Because in Steps~1 and~2 we were removing two cross-edges per one removed bad edge, the number of edges between $V_1$ and $V_2$ that were removed in Step~3 is $|V_1||V_2|-m-2b_1-2b_2\le t_2(n)-2b_1-2b_2$. Putting these inequalities together, we conclude that
 $$
 \pi_3(G)\le \underbrace{3b_1+3b_2}_{\mbox{\small Step 1\&2 }\triangle\mbox{\small s}}+\underbrace{2\frac{2\rho}{c n -\rho} b_1}_{\mbox{\small Step 3 bad edges}} + ~~~ \underbrace{2(t_2(n)-2b_1-2b_2)}_{\mbox{\small Step 3 cross-edges}}\le 2t_2(n).
 $$
 If $\pi_3(G)=2t_2(n)$, then $b_1=b_2=b_3=0$, $m=0$, and $v_1v_2=t_2(n)$, which clearly implies that $G\cong T_2(n)$, as required. \end{proof}

 \end{document}
 
 OLD KEY LEMMA
 
 \newpage
 \begin{lemma}\label{lm:key} There exists $n_1\in \mathbb{N}$ such that for any graph $G$ of order $n\ge n_1$ and minimum degree at least $(\ell(n)-\ell(n-1))/2$, we have
$\pi_3(G)\le \ell(n)$ with equality if and only if $G\in\C E_n$.\end{lemma}

\begin{proof}[Proof of Lemma \ref{lm:key}.] Choose small constants in this order $1\gg c\gg \delta\gg 1/n_1>0$. In particular, $n_1$ is sufficiently large to satisfy Corollary~\ref{co:stab} for this $\delta$. Let $G$ be an arbitrary graph of order~$n\ge n_1$ with $\pi_3(G)\ge \ell(n)$. By Corollary~\ref{co:stab}, $G$ is $\delta n^2$-close to either $T_2(n)$ or $K_n$. 

\textbf{Case 1:}
Suppose first that $G$ is $\delta n^2$-close to $T_2(n)$. We will show that $\pi_3(G)\le \pi_3(T_2(n))$ with equality if and only if $G\cong T_2(n)$, which is enough for proving the lemma in this case. In fact, this claim can be directly derived from a result of Gy\H ori~\cite{Gyori88} that a graph with $n$ vertices and $t_2(n)+k$ edges, where $t_2(n):= |E(T_2(n))|$, $n\to\infty$ and $k=o(n^2)$,
has at least $k-O(k^2/n^2)$ edge-disjoint triangles. More specifically, for each $\e>0$ there exists $\delta>0$ such that for large $n$ every $n$-vertex graph with $t_2(n)+k$ edges where $k\leq \delta n^2$ has at least $k-\e k^2/n^2$ edge-disjoint triangles. Since $G$ is $\delta n^2$-close to $T_2(n)$, then it must have at most $t_2(n)+\delta n^2$ edges. From this we have that $\pi_3(G)\leq 2(t_2(n))+k)-3(k-\e k^2/n^2)=2t_2(n)-k(1-3\e k/n^2)\leq 2t_2(n)$ for $\delta\ll \e\ll 1$. Equality is achieved only if $k=0$, that is, if $G\cong T_2(n)$. For completeness, we include a self-contained proof of this result in the \nameref{sec:appendix}.

\textbf{Case 2:} Now, suppose that $G$ is $\delta n^2$-close to $K_n$. Denote by $M=E(\O G)$ the set of pairs of vertices which are non-edges in $G$. Define 
$$
 U:=\{x\in V(G)\mid d_M(x)\ge cn\}.
 $$
 and let $W:=V(G)\setminus U$. By counting edges from $U$ to $V(G)$ we have  $|U|\cdot cn \le 2\delta n^2$ and hence $u:=|U|\le (2\delta/c)n$.
 
We decompose the edges of $G$ into triangles and edges in three Steps:
\begin{description}
\item[Step 1:] Consider all vertices $x\in U$ one by one and remove a maximum family of edge-disjoint triangles, each containing $x$ and two vertices from $W$. Denote the resulting graph induced on $W$ by $G'$.

\item[Step 2:] Remove a maximum collection of edge-disjoint triangles from $G'$.
\item[Step 3:] Decompose the rest of the graph into copies of $K_2$.
\end{description} 

First, we show that, up to parity, the triangles removed during Step 1 cover all $[U,W]$-edges in $G$.

\begin{claim} $|\Gamma_{G'}(u)\cap W|\leq 1$ for all $u\in U$.\end{claim}

\begin{proof}[Proof of claim.] 
Consider the moment when a vertex $x\in U$ is processed during Step 1.
Let $G''\subseteq G$ be the graph before the processing of $x$ starts.
By the minimum degree assumption, we have that $X:=\Gamma_G(x)\cap W$ contains at least $(n-2)/2-u+1\ge n/2-u$ vertices. For every vertex $y\in X$, the number of removed $[y,X]$-edges is at most $u$. Also, since $y\not\in U$, we have $d_M(y)<cn$. Thus, we have that
$$
\delta(G''[X])\ge |X|-1-u-cn >|X|/2.
$$
Then, by Dirac's theorem, $G''[X]$ has a Hamilton cycle. Pick alternating edges on this cycle, which combined with $x$ give the desired family of triangles.\end{proof}

Note that $G'$ has minimum degree at least $|W|-1-cn-u>(1-2c)|W|$. Let $S$ be the set of vertices of $G'$ of odd degree, denote $s:=|S|$ (note that $s$ is necessarily even), and let $p$ be the number of edges of $G'[W]$ that remain after Step 2. 

\begin{claim}\label{cl:p} We have $p\le \max\{s/2+r,6cn\}$, where $r\in\{0,1,2\}$ is the residue of $e(G')-s/2$ modulo $3$. In particular, $p\le (n-u)/2+2$.\end{claim}

\begin{proof}[Proof of claim.] First, suppose that $s>5cn$. Note that the minimum degree of $G'[S]$ is at least $s-2cn>s/2$. We aim to remove a matching from $G'[S]$ consisting of roughly $s/2$ edges, plus maybe some small graph (depending on the value of $r$), in order to obtain a triangle-divisible graph $G''$ (that is, such that $d(v)$ is even for all vertices $v\in G''$ and $3|e(G'')$), which is perfectly decomposable into triangles by Theorem~\ref{thm:decomp}.
Notice that $p$ counts the edges that we removed to make $G''$ triangle-divisible, so this gives $p\le s/2$ plus a small constant.\\
If $r=0$, by Dirac's theorem $G'[S]$ has a Hamilton cycle and thus a perfect matching $N$. The graph $G'':=G'-N$ has all degrees even and 3 divides $e(G'')$, which gives $p\le s/2$.\\
If $r=1$, let $v$ be any vertex in $S$, and let $\{v_1,v_2,v_3\}$ be three of its neighbours in $G'$. Again by Dirac's theorem we can find a perfect matching $N$ in $G'[S\backslash \{v,v_1,v_2,v_3\}]$, giving $p\le s/2+1$.\\ 
If $r=2$, let $xy$ be an edge in $W\backslash S$, and let $N_x$ and $N_y$ be the neighbourhoods of its endpoints in $S$ \comment{Y: This only works if $W\backslash S$ has an edge...}. Then $|N_x|,|N_y|\geq (1-2c)|W|-(|W|-s)>3cn$. In particular, we may choose a vertex $x'\in S$ which is a neighbour of $x$ and a vertex $y'\in S$ which is a neighbour of $y$. As before, we take $N$ to be a perfect matching on $S\backslash\{x',y'\}$, which gives $p\le s/2+2$. \comment{Instead do $K_{1,5}$ and a perfect matching. Still gives $p\le s/2+2$.}

Now suppose that $s\le 5cn$. Given any two vertices $x,y\in S$, their common neighbourhood in $W\backslash S$ has size at least $2((1-2c)|W|-s)-|W|\ge s/2$, so we can greedily find a collection of $s/2$ cherries (copies of $K_{1,2}$) whose degree-1 vertices partition $S$ and whose degree-2 vertices lie in $W\backslash S$. Extending up to two of these cherries to paths of length 3 and removing them leaves a triangle-divisible graph, giving $p\le s+2\le 6cn$.\end{proof}

It now remains to show that this decomposition has total cost at most $\ell(n)$.

Let $t_1$ and $t_2$ be the number of triangles removed respectively in Steps 1 and 2. 
By counting pairs of vertices inside $W$, we conclude that $t_1+3t_2+p\le {n-u\choose 2}$. Indeed, each triangle contributing to $t_1$ corresponds to a pair of vertices in $W$, and each triangle from Step 2 is contained in $W$ and hence corresponds to three pairs of vertices in $W$.

Since each vertex of $U$ has degree at most $(1-c)n$ in $G$, we also have that $t_1\le u(1-c)n/2$. 
Thus we obtain
 \begin{eqnarray*}
 \pi_3(G)&\le& 3t_1+3t_2+2{u\choose 2}+2p+2u\\
 &\le& 3t_1+\left({n-u\choose 2}-p-t_1\right)+2{u\choose 2}+2p+2u\\
  &\le& u(1-c)n +{n-u\choose 2}+2{u\choose 2}+p+2u\\
  &=& {n\choose 2} +\frac{3u^2}{2}+\frac{3u}{2}-cnu +p.
 \end{eqnarray*}

We now compare this bound with the conjectured maxima presented in Table \ref{ta:1}.

First, suppose that $n$ is even. In this case the larger value is achieved by $K_n$ and it is at least $\pi_3(K_n)\ge n^2/2={n\choose 2}+\frac n2$. Since $u\le (2\delta/c)n$, we have that $3u^2/2+\frac{3u}{2}\le cnu/2$ and so by Claim~\ref{cl:p}
$$\pi_3(G)-\pi_3(K_n)\le
-cnu/2+(n-u)/2+2-(n+u)/2.
$$ 
\comment{BL: Why $(n+u)/2?$	}
\[
\pi_3(G)-\pi_3(K_n)\le
-cnu/2+(n-u)/2+2-n/2 =  
-cnu/2-u/2+2 
\]

This is non-negative only if $u=0$, and since $G$ is extremal, all inequalities we used in upper-bounding $\pi_3(G)$ are tight. In particular, we get that $t_1=u(1-c)n/2=0$, and hence $e(G)=3t_2+p=\binom{n-u}{2}=\binom{n}{2}$. This immediately gives that $G\cong K_n$.

%
Now, suppose that $n$ is odd. In this case we have $\pi_3(T_2(n))-n/2\ge  \pi_3(K_n)-O(1)={n\choose 2}-O(1)$. Similarly to the previous case, $\pi_3(G)-\pi_3(T_2(n))\le -cnu/2+O(1)$, which again is non-negative only if $u=0$.
Hence
\begin{align*}
\pi_3(G) &\leq \frac{(n-s)(n-1) + s(n-2) - 2\cdot \max\{s/2+2,6cn\}}{2} +  2\cdot \max\{s/2+2,6cn\} \\
&= \binom{n}{2} - \frac{s}{2} + \max\{s/2+2,6cn\} \leq \binom{n}{2} + 2 + 6cn <  \binom{n}{2}  + n/4 < \pi_3(T_2(n)),
\end{align*}
a contradiction to the extremality of $G$. This finishes the proof of Lemma~\ref{lm:key}.
\end{proof}

\end{document}